\documentclass[a4paper,10pt,psamsfonts]{amsart}

\usepackage[utf8]{inputenc} 

\usepackage{amsmath}
\usepackage{amstext}
\usepackage{amsfonts}
\usepackage{amsthm}
\usepackage{amssymb}

\usepackage{hyperref}
\hypersetup{colorlinks}

\usepackage{color}

\newcommand{\R}{{\mathbb{R}}}
\newcommand{\E}{{\mathbb{E}}}
\newcommand{\N}{{\mathbb{N}}}
\newcommand{\D}{{\mathcal{D}}}
\newcommand{\F}{{\mathcal{F}}} 
\renewcommand{\P}{{\mathbf{P}}} 
\newcommand{\dom}{{\mathrm{dom}}}
\newcommand{\B}{{\mathcal{B}}}
\newcommand{\LB}{{\mathcal{L}}}

\newcommand{\diff}[1]{\,\mathrm{d}#1}

\newcommand{\one}{\mathbb{I}}

\newcommand{\Id}{\mathrm{Id}}

\theoremstyle{plain}

\newtheorem{definition}{Definition}[section]
\newtheorem{theorem}[definition]{Theorem}
\newtheorem{lemma}[definition]{Lemma}
\newtheorem{corollary}[definition]{Corollary}
\newtheorem{proposition}[definition]{Proposition}
\newtheorem{assumption}[definition]{Assumption}

\theoremstyle{definition}
\newtheorem{remark}[definition]{Remark}
\newtheorem{remarks}[definition]{Remarks}

\allowdisplaybreaks[1]
\begin{document}

\title[Consistency and Stability of a Milstein Scheme for Semilinear SPDE]
{Consistency and Stability\\of a Milstein-Galerkin finite element scheme\\
for semilinear SPDE}

\author[R.~Kruse]{Raphael Kruse}
\address{Raphael Kruse\\
ETH Z\"urich\\
Seminar f\"ur Angewandte Mathematik\\
R\"amistrasse 101\\
CH-8092 Z\"urich\\
Switzerland}
\email{raphael.kruse@sam.math.ethz.ch}

\keywords{semilinear SPDE, Milstein, Galerkin finite element methods, strong
convergence, spatio-temporal discretization, Spijker 
norm, bistability, consistency, two-sided error estimate} 
\subjclass[2010]{60H15, 65C30, 65M60, 65M70} 

\begin{abstract}
  We present an abstract concept for the error analysis of numerical
  schemes for semilinear stochastic partial differential equations (SPDEs) and
  demonstrate its usefulness by proving the strong convergence of a
  Milstein-Galerkin finite element scheme. By a suitable generalization of the
  notion of bistability from Beyn \& Kruse (DCDS B, 2010) to the semigroup
  framework in Hilbert spaces, our main result includes a
  two-sided error estimate of the spatio-temporal discretization. In an
  additional section we derive an analogous result for a Milstein-Galerkin
  finite element scheme with truncated noise. 
\end{abstract}

\maketitle

\section{Introduction}
\label{sec:intro}

The computational approximation of stochastic partial differential equations
(SPDEs) often turns out to be a very expensive and demanding task. One usually
has to combine numerical schemes for the temporal discretization of the interval
$[0,T]$ with Galerkin finite element methods for the spatial discretization as
well as truncation methods for the infinite dimensional noise. By the
combination of such schemes one then generates a sample path of the numerical
solution. If we are interested in the approximation of expected values of 
functionals of the solution, we have to repeat this procedure
several times in order to compute a decent Monte Carlo approximation.

For instance, let $X \colon [0,T]
\times \Omega \to H$ be a Hilbert-space valued stochastic process, which
denotes the solution to the given SPDE. Then our computational goal may be a
good approximation of the real number 
\begin{align}
  \label{eq1:number}
  \E[\varphi(X(T))],
\end{align}
where $\varphi \colon H \to \R$ is a sufficiently smooth mapping.

Before the upcoming of the multilevel Monte Carlo algorithm (MLMC)
\cite{giles2008a,kebaier2005}, it was common to purely focus on weakly
convergent schemes for the problem \eqref{eq1:number}. These schemes guarantee  
a good approximation of the distribution of $X$ and are then combined with a
standard Monte Carlo estimator to compute an approximation of
\eqref{eq1:number}.

In \cite{giles2008a} M.~Giles pointed out that the computational complexity of
problem \eqref{eq1:number} can drastically be reduced by the MLMC algorithm,
which distributes the most costly work of the Monte Carlo estimator to coarser
time grids, while relatively few samples need to be simulated of the smallest
and hence most costly temporal step size. But for this idea to work
one also needs to take the order of strong convergence into account. In
addition to an approximation of the distribution, a strongly
convergent scheme generates good pathwise approximations of
the solution $X$. For more details on strong and weak convergence we refer
to \cite{kloeden1999}. 

Additionally, M.~Giles showed in \cite{giles2008b}
that the usage of higher order strongly convergent schemes, such as the
Milstein method \cite{milstein1995}, further reduces the computational
complexity, although the order of weak convergence remains unchanged. While  
\cite{giles2008a,giles2008b,kebaier2005} are purely concerned with the finite
dimensional SODE problem, similar results also hold for solutions to SPDEs 
\cite{barth2012,barth2012a}.

Consequently, this observation has spurred the study of an infinite 
dimensional analogue of the Milstein scheme and first results have
been achieved for a temporal semidiscretization of linear SPDEs in 
\cite{lang2010, lang2012}. Afterwards, the Milstein scheme has been combined
with Galerkin finite element methods and extended to more general types of
driving noises in \cite{barth2011, barth2013}, while it was applied to
semilinear SPDEs in \cite{jentzen2010a}, but only with spectral Galerkin
methods. 

In this paper we apply the more general Milstein-Galerkin finite element
methods to the class of semilinear SPDEs studied in \cite{jentzen2010a}. Under
mildly relaxed assumptions on the nonlinearities we obtain slightly sharper
estimates of the error of strong convergence. For this we embed the scheme into
a more abstract framework and analyze the strong error with respect to the
notion of bistability and consistency, which originated from \cite{stummel1973}
and has been applied to SODEs for the first time in \cite{beynkruse2010,
kruse2011}. A key role is played by the choice of the so-called Spijker norm
\eqref{eq3:norm1} (see also \cite[p.~438]{hairer1993} and \cite{spijker1968,
spijker1971}), which is used to measure the local truncation error and results
into two-sided estimates of the error as shown in Theorem \ref{th1:Milconv}
below.

In forthcoming publications we show that the abstract concept is not only
useful for the analysis of the error of strong convergence for a broader class
of numerical schemes, but also has its merits in the weak error analysis as well
as in the analysis of an improved MLMC algorithm for SPDEs. 

In order to give a more detailed outline of the paper we first fix some
notation. Let $[0,T]$ be a finite time interval and $(H,( \cdot, \cdot )_H, \|
\cdot \|_H )$ and $(U, (\cdot, \cdot)_U, \| \cdot \|_U)$ be two separable real
Hilbert spaces. We denote by $(\Omega, \F, \P)$ a probability space which is
combined with a normal filtration $(\F_t)_{t \in [0,T]} \subset \F$ satisfying
the usual conditions.  
Then, let $(W(t))_{t\in [0,T]}$ be a cylindrical $Q$-Wiener
process in $U$ with respect to $( \mathcal{F}_t )_{ t \in [0,T] }$. Here, 
the given covariance operator $Q \colon U \to U$ is assumed to be bounded,
symmetric and positive semidefinite, but not necessarily of finite trace. 
For the definition of cylindrical $Q$-Wiener processes in $U$ we refer
to \cite[Ch.~2.5]{roeckner2007}. 

Next, we introduce the semilinear SPDE, whose solution we want to approximate.
Let $X \colon [0,T] \times \Omega \to H$ denote the
mild solution \cite[Ch. 7]{daprato1992} to the semilinear SPDE 
\begin{align}
  \begin{split}
    \mathrm{d}X(t) + \big[ AX(t) + f(X(t)) \big] \diff{t} &= g(X(t))
    \diff{W(t)}, \text{ for } 0 \le t \le T,\\ 
    X(0) &= X_0.
\end{split}
\label{eq1:SPDE}
\end{align}
Here, $-A \colon \dom(A) \subset H \to
H$ is the generator of an analytic semigroup $(S(t))_{t \ge 0}$ on $H$ and $f$
and $g$ denote nonlinear mappings which are Lipschitz continuous and 
smooth in an appropriate sense. In Section \ref{subsec:assumptions} we give a
precise formulation of our conditions on $A$, $f$, $g$ and $X_0$, which are also
sufficient for the existence and uniqueness of mild solutions $X$ (see also
Section \ref{subsec:mild}).

By definition \cite[Ch. 7]{daprato1992} the mild solution satisfies
\begin{align}
  X(t) = S(t) X_0 - \int_0^t S(t-\sigma) f(X(\sigma)) \diff{\sigma} + \int_0^t
  S(t-\sigma) g(X(\sigma)) \diff{W(\sigma)} 
  \label{eq1:mild}
\end{align}
$\P\text{-a.s.}$ for all $0 \le t \le T$.  

As our main example we have the following situation in mind: $H$ is the space
$L_2 (\D;\R)$ of square integrable functions, where $\mathcal{D} \subset \R^d$
is a bounded domain with smooth boundary $\partial \D$ or a convex domain with
polygonal boundary. Then, for example, let $-A$ be the Laplacian with
homogeneous Dirichlet boundary conditions. Much more extensive lists of
examples are given in \cite{jentzen2010a, jentzen2010b} and
\cite[Ch.~2.3]{kruse2013}. 

In order to introduce the \emph{Milstein-Galerkin finite element scheme}
we denote by $k \in (0,T]$ a given equidistant time step size with grid points
$t_n = nk$, $n = 1,\ldots,N_k$, and by $h \in (0,1]$ a parameter for the
spatial discretization. Then the Milstein scheme for the spatio-temporal
discretization of the SPDE \eqref{eq1:SPDE} is given by the recursion 
\begin{align}
  \label{eq4:Milstein}
  \begin{split}
    X_{k,h}(t_0) &= P_h X_0,\\
    X_{k,h}(t_n) &= X_{k,h}(t_{n-1}) - k \big[ A_h X_{k,h}(t_n) +
    P_h f(X_{k,h}(t_{n-1})) \big]\\
    &\quad + P_h g(X_{k,h}(t_{n-1})) \Delta_k W(t_{n})
    \\  
    &\quad + \int_{t_{n-1}}^{t_n} P_h g'(X_{k,h}(t_{n-1}))\Big[ 
    \int_{t_{n-1}}^{\sigma_1} g(X_{k,h}(t_{n-1})) \diff{W(\sigma_2)} \Big]
    \diff{W(\sigma_1)}
  \end{split}
\end{align}
for $n \in \{1,\ldots,N_k\}$, where $\Delta_k W(t_n) := W(t_n) - W(t_{n-1})$.
Here, $P_h$, $h \in (0,1]$, denotes the orthogonal projector onto the Galerkin
finite element space $V_h \subset H$ and $A_h$ is a discrete version of the
generator $A$. Together with some useful error estimates the operators of the
spatial approximation are explained in more detail in Section
\ref{subsec:Galerkin}.

In Section \ref{sec:numscheme} we introduce a class of abstract numerical
one-step schemes in Hilbert spaces and we develop our stability and consistency
analysis within this framework. We end up with a set of sufficient conditions
for the so-called bistability (see Definition \ref{def:stab}) and a
decomposition of the local truncation error. 

In Sections \ref{sec:Mil} and \ref{subsec:Milcons} we verify that the
scheme \eqref{eq4:Milstein} is indeed bistable and consistent (see Theorems
\ref{th:Milstab} and \ref{th4:cons}). These two properties together yield 
our main result (compare with Theorem \ref{th:lvlconv}):

\begin{theorem}
  \label{th1:Milconv}
  Suppose the spatial discretization fulfills Assumptions \ref{as:Vh} and
  \ref{as:Vh2}. If Assumptions \ref{as4:initial} to \ref{as4:g} are satisfied
  with $p \in [2,\infty)$ and $r \in [0,1)$, then there exists a constant $C$
  such that
  \begin{align}
    \label{eq1:twosided}
    \frac{1}{C} \big\| \mathcal{R}_k [ X|_{\mathcal{T}_k}] \big\|_{-1,p} \le 
    \max_{0 \le n \le N_k} \big\| X_{k,h}(t_n) - X(t_n)
    \big\|_{L_p(\Omega;H)} \le C \big\| \mathcal{R}_k [ X|_{\mathcal{T}_k}]
    \big\|_{-1,p}.
  \end{align}
  In particular, from the estimate of the local truncation error it follows
  that
  \begin{align}
    \label{eq1:error}
    \max_{0 \le n \le N_k} \big\| X_{k,h}(t_n) - X(t_n)
    \big\|_{L_p(\Omega;H)} \le C \big( h^{1+r} + k^{\frac{1+r}{2}} \big) 
  \end{align}
  for all $h \in (0,1]$ and $k \in (0,T]$, where $X_{k,h}$
  denotes the grid function generated by the scheme \eqref{eq4:Milstein} and
  $X$ is the mild solution to \eqref{eq1:SPDE}. 
\end{theorem}

At this point, for a better understanding of Theorem \ref{th1:Milconv}, let us
explain the different objects appearing in its formulation. First,
Assumptions \ref{as:Vh} and \ref{as:Vh2} are concerned with the spatial
discretization. In our main example above, they are usually 
satisfied for the standard piecewise linear finite element method. 

Roughly speaking, the other assumptions determine the spatial regularity of the
solution, which is measured by the parameter $r \in [0,1)$ in terms of
fractional powers of the operator $A$. This parameter also mainly controls the
order of convergence. 

Then, in \eqref{eq1:twosided} the residual operator $\mathcal{R}_k$ of the
numerical scheme \eqref{eq4:Milstein} appears. It characterizes 
\eqref{eq4:Milstein} in the sense that $\mathcal{R}_k[Z] = 0$ if and only if
the $H$-valued grid function $Z$ coincides with $X_{k,h}$. We therefore use the
residual operator in order to determine how far the exact solution $X$
(restricted to the time grid) differs from the numerical solution $X_{k,h}$.
This residual is called local truncation error and, if measured in terms of
the stochastic Spijker norm \eqref{eq3:norm1}, it can be used to estimate the
strong error from above and below (compare further with Section
\ref{sec:numscheme}).

In \cite{kruse2011} two-sided error estimates of the form \eqref{eq1:twosided}
have been used to prove the maximal order of convergence of all It\^o-Taylor
schemes. However, this question is not discussed in this paper and it is
subject to future research if a similar result can be derived for the
Milstein-Galerkin finite element scheme.

Further, we note that the order of convergence in \eqref{eq1:error} is slightly
sharper than in \cite{jentzen2010a}, where the corresponding result
contains a small order reduction of the form $1 + r - \epsilon$ for arbitrary
$\epsilon > 0$. As in \cite{kruse2012} this order reduction is avoided by the
application of Lemma \ref{lem:Fkh2}, which contains sharp integral versions
of estimates for the Galerkin finite element error operator. 

In practice, the scheme \eqref{eq4:Milstein} can seldomly be implemented
directly on a computer due to the fact that the space $U$ and thus also the
noise $W$ is probably of high or infinite dimension. In our final Section
\ref{sec:Noise} we discuss the stability and consistency of a variant of
\eqref{eq4:Milstein}, which incorporates a spectral approximation of the Wiener
process. This approach has already been studied by several
authors in the context of Milstein schemes for SPDEs, for instance in
\cite{barth2011, barth2013, jentzen2010a}. With Theorem \ref{th6:conv} we
obtain an extended version of Theorem \ref{th1:Milconv}, which also takes the
noise approximation into account.

\section{Preliminaries}
\label{sec:prelim}

\subsection{Main Assumptions}
\label{subsec:assumptions}

In this subsection we give a precise formulation of our assumptions on the SPDE
\eqref{eq1:SPDE}. The first one is concerned with the linear operator. 

\begin{assumption}
  \label{as4:linop}
  The linear operator $A \colon \dom(A) \subset H \to H$ is densely defined,
  self-adjoint and positive definite with compact inverse.  
\end{assumption}

As in \cite[Ch.~2.5]{pazy1983} it follows from Assumption \ref{as4:linop} that
the operator $-A$ is the generator of an analytic semigroup $(S(t))_{t \in
[0,T]}$ on $H$. There also exists an increasing, real-valued sequence
$(\lambda_i)_{i \in \N}$ with $\lambda_i > 0$, $i \in \N$, and $\lim_{i \in \N}
\lambda_i = \infty$ and an 
orthonormal basis $(e_i)_{i \in \N}$ of $H$ such that $A e_i = \lambda_i e_i$
for every $i \in \N$.

Further, we recall the definition of fractional powers of $A$ from
\cite[Ch.~B.2]{kruse2013}. For any $r \ge 0$ the 
operator $A^{\frac{r}{2}} \colon \dom( A^{\frac{r}{2}} ) \subset H \to H$ is
defined by
\begin{align*}
  A^{\frac{r}{2}} x := \sum_{j = 1}^\infty \lambda_j^{\frac{r}{2}} (x, e_j) e_j
\end{align*}
for all 
\begin{align*}
  x \in \dom(A^{\frac{r}{2}}) = \Big\{ x \in H \, \colon \, \sum_{j = 1}^\infty
  \lambda_j^r (x, e_j)^2 < \infty \Big\}. 
\end{align*}
Endowed with the inner product $(\cdot, \cdot)_r := (A^{\frac{r}{2}} \cdot,
A^{\frac{r}{2}} \cdot )$ and norm $\| \cdot \|_{r} := \|
A^{\frac{r}{2}} \cdot \|$ the spaces $\dot{H}^{r}:=
\dom(A^{\frac{r}{2}})$ become separable Hilbert spaces. 

In addition, we define the spaces $\dot{H}^{-r}$ with negative exponents
as the dual spaces of $\dot{H}^{r}$, $r > 0$. In this case it follows 
from \cite[Th.~B.8]{kruse2013} that the elements of $\dot{H}^{-r}$ can be
characterized by 
\begin{align*}
  \dot{H}^{-r} = \Big\{ x = \sum_{j =1}^{\infty} x_j e_j \, \colon \,
  (x_j)_{j \in \N} \subset \R,\; \text{ with } \sum_{j = 1}^\infty 
  \lambda_j^{-r} x_j^2 < \infty \Big\},  
\end{align*}
where the equality is understood to be isometrically isomorphic and the norm in 
$\dot{H}^{-r}$ can be computed by $\| x \|_{-r} = \| A^{-\frac{r}{2}} x \|$.
Here, we set
\begin{align*}
  A^{-\frac{r}{2}} x = \sum_{j = 1}^{\infty} \lambda_j^{-\frac{r}{2}} x_j e_j,
  \quad \text{ for all } x = \sum_{j = 1}^\infty x_j e_j \in \dot{H}^{-r}.
\end{align*}
For the formulation of the remaining assumptions let parameter values $p \in
[2,\infty)$ and $r \in [0,1)$ be given.

\begin{assumption}
  \label{as4:initial}
  The random variable $X_0 \colon \Omega \to \dot{H}^{1+r}$ is
  $\F_0/\mathcal{B}(\dot{H}^{1+r})$-measur\-able. In addition, it holds
  \begin{align*}
    \E \big[ \| X_0 \|^{2p}_{1+r} \big] < \infty. 
  \end{align*}
\end{assumption}

The next assumption is concerned with the nonlinear mapping $f \colon H \to
\dot{H}^{-1+r}$ in \eqref{eq1:SPDE}. 

\begin{assumption}
  \label{as4:f} 
  The mapping $f \colon H \to \dot{H}^{-1+r}$
  is continuously Fr\'echet differentiable.
  In addition, there exists a
  constant $C_f$ such that $\| f(0) \|_{-1+r} \le C_f$ and 
  \begin{align*}
    \sup_{x \in H} \| f'(x) \|_{\LB(H;\dot{H}^{-1+r})} \le C_f, 
  \end{align*}
  as well as
  \begin{align}
    \label{eq4:flip}
    \begin{split}
      \| f(x_1) - f(x_2) \|_{-1+r} &\le C_f \| x_1 - x_2 \|,\\
      \| f'(x_1) - f'(x_2) \|_{\LB(H,\dot{H}^{-1+r})} &\le C_f \| x_1 - x_2 \|,
    \end{split}
  \end{align}
  for all $x_1, x_2 \in H$. 
\end{assumption}

The last assumption deals with the nonlinear mapping $g$ in the
stochastic integral part of \eqref{eq1:SPDE}. As in \cite{daprato1992,
roeckner2007} we denote the so-called Cameron-Martin space by $U_0 :=
Q^{\frac{1}{2}}(U)$, which together with the inner 
product $(u_0,v_0)_{U_0} := ( Q^{-\frac{1}{2}} u_0, Q^{-\frac{1}{2}} v_0 )_U$
for $u_0, v_0 \in U_0$ becomes an Hilbert space. Here $Q^{-\frac{1}{2}}$ 
denotes the pseudoinverse \cite[App.~C]{roeckner2007} of $Q^{\frac{1}{2}}$.

Then, by $\LB_2(H_1,H_2) \subset \LB(H_1,H_2)$ we denote the space of all
Hilbert-Schmidt operators $L \colon H_1 \to H_2$ between two separable Hilbert
spaces $H_1$ and $H_2$. Together with the inner product 
\begin{align*}
  ( L_1, L_2)_{\LB_2(H_1,H_2)} = \sum_{j = 1}^\infty \big( L_1 \psi_j, L_2
  \psi_j \big)_{H_2}, 
\end{align*}
where $(\psi_j)_{j \in \N}$ is an arbitrary orthonormal basis of $H_1$,
the set $\LB_2(H_1,H_2)$ becomes an Hilbert space. We recall the abbreviations
$\LB_2^0 := \LB_2(U_0,H)$ and $\LB_{2,r}^0 :=  \LB_2(U_0, \dot{H}^r)$ from
\cite{kruse2012} and refer to \cite[App.~B]{roeckner2007} for a short review on
Hilbert-Schmidt operators. 

\begin{assumption}
  \label{as4:g}
  Let the mapping $g \colon H \to \LB_2^0$ be continuously Fr\'echet
  differentiable. In addition, there exists a
  constant $C_g$ such that $\| g(0) \|_{\LB_2^0} \le C_g$ and 
  \begin{align*}
    \sup_{x \in H} \| g'(x) \|_{\LB(H;\LB_2^0)} \le C_g,
  \end{align*}
  as well as
  \begin{align}
    \label{eq4:glip}
    \begin{split}
      \| g(x_1) - g(x_2) \|_{\LB_2^0} &\le C_g \| x_1 - x_2\|,\\
      \| g'(x_1) - g'(x_2) \|_{\LB(H,\LB_2^0)} & \le C_g \| x_1 - x_2\|, \\
      \| g'(x_1)g(x_1) - g'(x_2)g(x_2) \|_{\LB_2(U_0,\LB_2^0)} &\le C_g \| x_1 -
      x_2\|, 
    \end{split}
  \end{align}
  for all $x_1, x_2 \in H$.

  Further, the mapping $g \colon H \to \LB_2^0$
  satisfies $g(x) \in \LB_{2,r}^0$ and
  \begin{align}
    \label{eq4:glin}
    \| g(x) \|_{\LB_{2,r}^0} \le C_g \big(1 +  \| x \|_r \big)
  \end{align}
  for all $x \in \dot{H}^r$.
\end{assumption}

\begin{remark}
  It is straightforward to generalize most of the results and techniques, which
  we develop in this paper, to the case when $f$ and $g$ are allowed to 
  also depend on $t\in [0,T]$ and $\omega \in \Omega$. For example, this has
  been done for the linearly implicit Euler-Maruyama method in \cite{kruse2013}. 
\end{remark}

\subsection{Existence, uniqueness and regularity of the mild solution}
\label{subsec:mild}

Under the assumptions of Subsection \ref{subsec:assumptions}, 
there exists a unique (up to modification) mild solution 
$X \colon [0,T] \times \Omega \to H$ to \eqref{eq1:SPDE} of
the form \eqref{eq1:mild}. A proof for this is found, for instance, in
\cite[Ch.~2.4]{kruse2013} (based on the methods from
\cite[Th.~1]{jentzen2010b}).

Furthermore, it holds true that for all $s \in [0,r+1]$, where $r \in 
[0,1)$ and $p \in [2,\infty)$ are given by Assumptions \ref{as4:initial} to
\ref{as4:g}, we have 
\begin{align}
  \label{eq2:reg}
  \sup_{t \in [0,T]} \E\big[ \| X(t) \|^{2p}_{s} \big]  < \infty
\end{align}
and there exists a constant $C$ such that
\begin{align}
  \label{eq2:hoelder}
  \big(\E \big[ \| X(t_1) -
  X(t_2) \|_{s}^{2p} \big] \big)^{\frac{1}{2p}} \le C |t_1
  -t_2|^{\min(\frac{1}{2},\frac{r+1 - s}{2})}
\end{align}
for all $t_1,t_2 \in [0,T]$. These regularity results have been proved in
\cite[Th.~1]{jentzen2010b} and \cite{kl2010a}.

\subsection{A Burkholder-Davis-Gundy type inequality}
Burkholder-Davis-Gundy-type inequalities are frequently used
to estimate higher moments of stochastic integrals. The version in
Proposition \ref{prop:stoch_int} is a special case of
\cite[Lem.~7.2]{daprato1992}. 

\begin{proposition}
  \label{prop:stoch_int} 
  For any $p \in [2,\infty)$, $0 \le \tau_1 < \tau_2 \le T$, and for any 
  predictable stochastic process $\Psi \colon [0,T] \times \Omega \to \LB_2^0$,
  which satisfies
  \begin{align*}
    \Big( \int_{\tau_1}^{\tau_2}
    \big\| \Psi(\sigma) \big\|^2_{L_p(\Omega;\LB_2^0)} \diff{\sigma}
    \Big)^{\frac{1}{2}} < \infty, 
  \end{align*}  
  we have
  \begin{align*}
    \Big\| \int_{\tau_1}^{\tau_2} \Psi(\sigma) \diff{W(\sigma)}
    \Big\|_{L_p(\Omega;H)}
    &\le C(p) \Big( \E \Big[ \Big( \int_{\tau_1}^{\tau_2}
    \big\| \Psi(\sigma) \big\|^2_{\LB_2^0} \diff{\sigma}
    \Big)^{\frac{p}{2}} \Big] \Big)^{\frac{1}{p}} \\
    &\le C(p)\, \Big( \int_{\tau_1}^{\tau_2}
    \big\| \Psi(\sigma) \big\|^2_{L_p(\Omega;\LB_2^0)} \diff{\sigma}
    \Big)^{\frac{1}{2}}. 
  \end{align*}
  Here the constant can be chosen to be
  \begin{align*}
    C(p) = \left( \frac{p}{2} ( p - 1)
    \right)^{\frac{1}{2}} \left( \frac{p}{p - 1}
    \right)^{(\frac{p}{2} - 1)}.
  \end{align*}
\end{proposition}

\begin{proof}
  Under the given assumptions on $\Psi$ it follows that
  \begin{align*}
    \Big( \E \Big[ \Big( \int_{\tau_1}^{\tau_2}
    \big\| \Psi(\sigma) \big\|^2_{\LB_2^0} \diff{\sigma}
    \Big)^{\frac{p}{2}} \Big] \Big)^{\frac{1}{p}}
    &= \Big\| \int_{\tau_1}^{\tau_2}
    \big\| \Psi(\sigma) \big\|^2_{\LB_2^0} \diff{\sigma} 
    \Big\|_{L_{p/2}(\Omega;\R)}^{\frac{1}{2}}\\
    &\le \Big( \int_{\tau_1}^{\tau_2}
    \big\| \Psi(\sigma) \big\|^2_{L_p(\Omega;\LB_2^0)} \diff{\sigma}
    \Big)^{\frac{1}{2}} < \infty.
  \end{align*}  
  Therefore, the stochastic integral is well-defined and
  \cite[Lem.~7.2]{daprato1992} yields 
  \begin{align*}
    \Big\| \int_{\tau_1}^{\tau_2} \Psi(\sigma) \diff{W(\sigma)}
    \Big\|_{L_p(\Omega;H)} &\le C(p) \Big( \E \Big[ \Big(
    \int_{\tau_1}^{\tau_2}  \big\| \Psi(\sigma) \big\|^2_{\LB_2^0}
    \diff{\sigma} \Big)^{\frac{p}{2}} \Big] \Big)^{\frac{1}{p}}\\
    &\le C(p) \Big( \int_{\tau_1}^{\tau_2}
    \big\| \Psi(\sigma) \big\|^2_{L_p(\Omega;\LB_2^0)} \diff{\sigma}
    \Big)^{\frac{1}{2}},    
  \end{align*}  
  which are the asserted inequalities.
\end{proof}

\subsection{Galerkin finite element methods}
\label{subsec:Galerkin}

In this subsection we recall the most important elements of 
Galerkin finite element methods. For a more detailed review we refer 
to \cite{kruse2012,kruse2013}, which in turn are based on \cite[Ch.~2, 3 and
7]{thomee2006}.  

Our starting point is a sequence $(V_h)_{h \in (0,1]}$ of finite dimensional 
subspaces of $\dot{H}^1$. Here, the parameter $h \in (0,1]$ controls
the dimension of $V_h$, which usually increases as
$h$ decreases. For smaller values of $h$ we therefore expect to find better
approximations of smooth elements of $H$ within $V_h$.

Then, for every $h \in (0,1]$ the 
\emph{Ritz projector} $R_h \colon \dot{H}^{1} \to V_h$ is the orthogonal
projector onto $V_h$ with respect to the inner product $(\cdot,\cdot)_1$ and
given by 
\begin{align*}
  \big( R_h x, y_h \big)_1 = \big( x, y_h \big)_1 \quad \text{for all } x \in
  \dot{H}^{1}, \, y_h \in V_h.
\end{align*}
The following assumption ensures that the spaces
$(V_h)_{h \in (0,1]}$ contain good approximations of all
elements in $\dot{H}^1$ and 
$\dot{H}^2$, respectively. It is formulated in terms of
the Ritz projector and closely related to the spatial approximation of the
elliptic problem $Au = f$ as noted in \cite[Rem.~3.4]{kruse2013}. Compare also
with \cite[(ii) on p.~31 and (2.25)]{thomee2006}). 

\begin{assumption}
  \label{as:Vh}
  Let a sequence $(V_h)_{h \in (0,1]}$ of finite dimensional subspaces of
  $\dot{H}^1$ be given such that there exists a constant $C$ with 
  \begin{align}
    \label{eq3:Rh}
    \big\| R_h x - x \big\| \le C h^s \| x \|_{s} \text{ for all } x \in
    \dot{H}^s, \; s\in\{1,2\}, \; h \in(0,1].
  \end{align}
\end{assumption}

Another important operator is the linear mapping $A_h \colon V_h \to V_h$,
which denotes a discrete version of $A$. For
a given $x_h \in V_h$ we define $A_h x_h \in V_h$ by the representation theorem
through the relationship 
\begin{align*}
  (x_h, y_h)_{1} = (A_h x_h, y_h) \quad \text{for all } y_h \in V_h.
\end{align*}
It directly follows that $A_h$ is self-adjoint and positive definite on $V_h$.

Finally, we denote by $P_h \colon \dot{H}^{-1} \to V_h$ the (generalized)
orthogonal projector onto $V_h$ with respect to the inner product in $H$.
As in \cite{chrysafinos2002} the projector $P_h$ is defined by
\begin{align*}
  (P_h x, y_h) = ( A^{-\frac{1}{2}} x, A^{\frac{1}{2}} y_h ) \quad \text{for
  all } x \in \dot{H}^{-1}, y_h \in V_h.
\end{align*}
After having introduced all operators for the spatial approximation we recall
the following discrete negative norm estimate from \cite[(3.7)]{larsson2006}
\begin{align}
  \label{eq2:discnorm}
  \begin{split}
    \| A_h^{-\frac{1}{2}} P_h x \| &= \sup_{z_h \in V_h}
    \frac{\big|(A_h^{-\frac{1}{2}}P_h x,z_h) \big|}{\| z_h \|} =
    \sup_{z_h \in V_h} \frac{\big|( P_h x,A_h^{-\frac{1}{2}}  z_h) \big|}{\|
    z_h \|} \\
    &= \sup_{z_h' \in V_h} \frac{\big|\langle x,z_h' \rangle
    \big|}{\|A_h^{\frac{1}{2}} z_h' \|} \le \sup_{z_h' \in V_h}
    \frac{\|x \|_{-1} \| z_h' \|_{1}}{ \|A_h^{\frac{1}{2}} z_h' \| } =  \|x
    \|_{-1} 
  \end{split}
\end{align}
for all $x\in \dot{H}^{-1}$.

The remainder of this subsection lists some error estimates for 
spatio-temporal Galerkin finite element approximations of the linear Cauchy
problem 
\begin{align}
  \label{eq2:linprob}
  \frac{\mathrm{d}}{\diff{t}} u(t) + A u(t) = 0,\quad t \in [0,T], \quad
  u(0) = x \in H. 
\end{align}
In terms of the semigroup $(S(t))_{t \in [0,T]}$ generated by $-A$, the
solution to \eqref{eq2:linprob} is given by $u(t) = S(t) x$ for all $t \in
[0,T]$.  

Let $k \in (0,T]$ be a given equidistant time step size. We define $N_k \in \N$
by $k N_k \le T < k(N_k + 1)$ and denote the set of all temporal grid points by
$\mathcal{T}_k := \{ t_n \, : \, n = 0,1,\ldots,N_k\,\}$ with $t_n = 
kn$. Then, we combine the spatially discrete operators with a backward Euler
scheme and obtain the spatio-temporal Galerkin finite element approximation
$u_{k,h} \colon \mathcal{T}_k \to V_h$ of \eqref{eq2:linprob}, which
is given by the recursion 
\begin{align}
  \label{eq2:Euler}
  \begin{split}
    u_{k,h}(t_0) &= P_h x,\\
    u_{k,h}(t_n) + k A_h u_{k,h}(t_n) &= u_{k,h}(t_{n-1}), \quad n =
    1,\ldots,N_k,
  \end{split}
\end{align}
for $h \in (0,1]$ and $k \in (0,T]$.
Equivalently, we may write $u_{k,h}(t_n) = S_{k,h}^n P_h u_0$ with
$S_{k,h} = ( I + k A_h)^{-1}$ for all $n \in \{0,\ldots,N_k\}$.

Similar to the analytic semigroup $(S(t))_{t \in [0,T]}$, the discrete operator
$S_{k,h}$ has the following smoothing property
\begin{align}
  \label{eq2:discsmoothing}
  \big\| A_h^{\rho} S_{k,h}^{-j} x_h \big\| =
  \big\| A_h^\rho ( \Id_H + k A_h )^{-j} x_h \big\| \le C t_{j}^{-\rho} \| x_h
  \|
\end{align}
for all $j \in \{1,\ldots,N_k\}$, $x_h \in V_h$, $k \in (0,T]$ and $h \in
(0,1]$. Here the constant $C = C(\rho)$ is independent of $h,k$ and $j$. For a
proof of \eqref{eq2:discsmoothing} we refer to \cite[Lem.~7.3]{thomee2006}. 

For the error analysis in Section \ref{sec:Mil}
it will be convenient to introduce the continuous time error operator between
\eqref{eq2:linprob} and \eqref{eq2:Euler}
\begin{align}
  \label{eq2:errOp}
  F_{k,h}(t) := S_{k,h}(t) P_h - S(t), \quad t \in [0,T),
\end{align}
where 
\begin{align}
  \label{eq2:defSkh}
  S_{k,h}(t) := ( \Id_H + k A_h )^{-j}, \quad \text{if } t \in [t_{j-1}, t_j)
  \text{ for } j\in\{1,2,\ldots\,N_k\} .
\end{align}
The mapping $t \mapsto S_{k,h}(t)$, and hence $t \mapsto F_{k,h}(t)$, is right
continuous with left limits. A simple consequence
of \eqref{eq2:discsmoothing} and \eqref{eq2:discnorm} are the inequalities
\begin{align}
  \label{eq2:stabSkh}
  \big\| S_{k,h}(t) P_h x \big\| \le C \big\| x \big\|  \quad \text{ for all }
  x \in H,
\end{align}
and
\begin{align}
  \label{eq2:smoothSkh}
  \big\| S_{k,h}(t) P_h x \big\| = \big\| A_h^{\frac{1}{2}} ( \Id_H + k A_h
  )^{-j}  A_h^{-\frac{1}{2}} P_h x \big\| \le C t_{j}^{-\frac{1}{2}} \big\| x 
  \big\|_{-1} \le C t^{-\frac{1}{2}} \big\| x
  \big\|_{-1},  
\end{align}
which hold for all $x \in \dot{H}^{-1}$, $h \in (0,1]$, $k \in (0,T]$ and $t >
0$ with $t \in [t_{j-1}, t_j)$, $j = 1,2,\ldots$. For both inequalities the
constant $C$ can be chosen to be independent of $h \in (0,1]$ and $k \in
(0,T]$.

The next lemma provides several estimates for the error operator $F_{k,h}$ with
non-smooth initial data. Most of the results are well-known and are found in
\cite[Ch.~7]{thomee2006}. The missing cases have been proved in
\cite[Lem.~3.12]{kruse2013}.

\begin{lemma}
  \label{lem:Fkh1}
  Under Assumption \ref{as:Vh} the following estimates hold true:

  (i) Let $0 \le \nu \le \mu \le 2$. Then there exists a constant $C$ such that
  \begin{align*}
    \big\| F_{k,h}(t) x \big\| \le C \big( h^{\mu} + k^{\frac{\mu}{2}} \big)
    t^{-\frac{\mu - \nu}{2}} \big\| x \big\|_{\nu} \text{ for all } x \in
    \dot{H}^\nu, \; t \in (0,T), \; h, k \in (0,1].
  \end{align*}

  (ii) Let $0 \le \rho \le 1$. Then there exists a constant $C$ such that
  \begin{align*}
    \big\| F_{k,h}(t) x \big\| \le C t^{-\frac{\rho}{2}} \big\| x \big\|_{-\rho}
    \text{ for all } x \in \dot{H}^{-\rho}, \; t \in (0,T), \; h, k \in (0,1].
  \end{align*}

  (iii) Let $0 \le \rho \le 1$. Then there exists a constant $C$ such that
  \begin{align*}
    \big\| F_{k,h}(t) x \big\| \le C \big( h^{2- \rho} + k^{\frac{2 - \rho}{2}}
    \big) t^{-1} \big\| x
    \big\|_{-\rho} \text{ for all } x \in \dot{H}^{-\rho}, \; t \in (0,T), \;
    h, k \in (0,1]. 
  \end{align*}
\end{lemma}

The next assumption is concerned with the stability of the orthogonal projector
$P_h$ with respect to the norm $\| \cdot \|_{1}$. It only appears in the proof
of Lemma \ref{lem:Fkh2} as shown in \cite[Lem.~3.13]{kruse2013}. 

\begin{assumption}
  \label{as:Vh2}
  Let a family $(V_h)_{h \in (0,1]}$ of finite dimensional subspaces of
  $\dot{H}^1$ be given such that there exists a constant $C$ with
  \begin{align}
    \| P_h x \|_{1} \le C \| x \|_{1} \quad \text{ for all } x \in \dot{H}^1,
    \; h \in (0, 1].
    \label{eq3:stabPh}
  \end{align}  
\end{assumption}

The last lemma of this section is concerned with sharper integral versions of
the error estimate in Lemma \ref{lem:Fkh1} \emph{(i)} and \emph{(iii)}. A proof
is given in \cite[Lem.~3.13]{kruse2013}. 

\begin{lemma}
  \label{lem:Fkh2}
  Let $0 \le \rho \le 1$. Under Assumption \ref{as:Vh} the operator $F_{k,h}$
  satisfies the following estimates.

  (i) There exists a constant $C$ such that
  \begin{align*}
    \Big\| \int_{0}^{t} F_{k,h}(\sigma) x \diff{\sigma} \Big\| \le C \big(
    h^{2 - \rho} + k^{\frac{2 - \rho}{2}} \big) \big\| x \big\|_{-\rho} 
  \end{align*}
  for all $x \in \dot{H}^{-\rho}$, $t > 0$, and $h, k \in (0,1]$.

  (ii) Under the additional Assumption \ref{as:Vh2} there exists a constant $C$
  such that 
  \begin{align*}
    \Big( \int_{0}^{t} \big\| F_{k,h}(\sigma) x \big\|^2 \diff{\sigma}
    \Big)^{\frac{1}{2}} \le C \big( h^{1 + \rho} + k^{\frac{1 + \rho}{2}} \big)
    \big\| x \big\|_{\rho} 
  \end{align*}
  for all $x \in \dot{H}^{\rho}$, $t > 0$, and $h, k \in (0,1]$.
\end{lemma}

\section{Stability and consistency of numerical one-step schemes}
\label{sec:numscheme}

This section contains the somewhat more abstract framework of the convergence
analysis. We generalize the notion of stability and consistency from
\cite{beynkruse2010, kruse2011} to Hilbert spaces and we derive a set of
sufficient conditions for the so-called bistability. Finally, a decomposition of
the local truncation error gives a blueprint for the proof of consistency of
the Milstein-Galerkin finite element scheme in Section \ref{subsec:Milcons}.

\subsection{Definition of the abstract one-step scheme}
\label{subsec:numschemedef}

As above, let $\mathcal{T}_k := \{t_n \, : \, n  = 0,1,\ldots,N_k \}$ be  
the set of temporal grid points for a given equidistant time step
size $k \in (0,T]$ and recall that $N_k \in \N$ is given by $N_k k \le T < (N_k
+ 1)k$.  

The first important ingredient, which determines the numerical scheme, is a
family of bounded linear operators $S_k \colon H \to H$, $k \in (0,T]$, which
are supposed to approximate the semigroup $S(t)$, $t \in [0,T]$, in a suitable
sense. 

Further, for the definition of the second ingredient, let us
introduce the set $\mathbb{T} \subset [0,T) \times (0,T]$, which is given by 
\begin{align*}
  \mathbb{T} := \{ (t,k) \in [0,T) \times (0,T] \; : \; t + k \le T \}. 
\end{align*}
The so-called \emph{increment function} is a mapping
$\Phi \colon H \times \mathbb{T} \times \Omega \to H$ with the property that
for every  $(t,k) \in \mathbb{T}$ the mapping 
$(x,\omega) \mapsto \Phi(x,t,k)(\omega)$ is measurable with respect to $\B(H)
\otimes \F_{t + k}/\B(H)$.  

Then, for every $k \in (0,T]$ the discrete time stochastic process   
$X_k \colon \mathcal{T}_k \times \Omega \to H$, 
is given by the recursion 
\begin{align}
  \label{eq3:scheme}
  \begin{split}
    X_k(t_0) &:= \xi,\\
    X_k(t_n) &:= S_k X_k(t_{n-1}) + \Phi\big( X_k(t_{n-1}), t_{n-1}, k \big)
  \end{split}
\end{align}
for every $n \in \{1,\ldots,N_k\}$,
where $\xi \colon \Omega \to H$, is an $\F_{t_0}/\B(H)$-measurable
random variable representing the initial value of the numerical scheme. It
follows directly that $X_k(t_n)$ is $\F_{t_n}/\B(H)$-measurable for all $n \in
\{1,\ldots,N_k\}$. 

In Section \ref{sec:Mil} we show how the Milstein Galerkin
finite element scheme fits into the framework of \eqref{eq3:scheme}.

After having introduced the abstract numerical scheme, we recall the
cornerstones of the stability and consistency concept for one-step methods from
\cite{beynkruse2010,kruse2011}. First, let us introduce the family of 
linear spaces of adapted, $p$-integrable grid functions
\begin{align*}
  \mathcal{G}_p(\mathcal{T}_k) := \big\{ Z \colon \mathcal{T}_k \times \Omega
  \to H \; : \; Z(t_n) \in L_p(\Omega,\mathcal{F}_{t_n},\P;H) \text{ for all }
  n \in \{0,1, \ldots,N_k \} \big\} 
\end{align*}
for all $p \in [2,\infty)$ and $k \in (0,T]$. The spaces
$\mathcal{G}_p(\mathcal{T}_k)$ are endowed with the two norms
\begin{align}
  \| Z \|_{0,p} := \max_{n \in \{0,\ldots,N_k\}} \big\| Z(t_n)
  \big\|_{L_p(\Omega;H)}
  \label{eq3:norm0}
\end{align}
and
\begin{align}
  \| Z \|_{-1,p} := \| Z(t_0) \|_{L_p(\Omega;H)} + \max_{n \in
  \{1,\ldots,N_k\}} \Big\| \sum_{j = 1}^n S_k^{n-j} Z(t_j)
  \Big\|_{L_p(\Omega;H)} 
  \label{eq3:norm1}
\end{align}
for all $Z \in \mathcal{G}_p(\mathcal{T}_k)$. The norm $\| \cdot \|_{-1,p}$
is called \emph{(stochastic) Spijker norm} and known to result into sharp
and two-sided estimates of the error of convergence, see for example
\cite[p.~438]{hairer1993} as well as
\cite{beynkruse2010,spijker1968,spijker1971,stummel1973}. 

Next, for $p\in[2,\infty)$, let us define the family of nonlinear operators
$\mathcal{R}_k \colon \mathcal{G}_p(\mathcal{T}_k) \to
\mathcal{G}_p(\mathcal{T}_k)$, which for $k \in (0,T]$ are given by  
\begin{align}
  \begin{split}
    \mathcal{R}_k[Z](t_0) &= Z(t_0) - \xi, \\
    \mathcal{R}_k[Z](t_n) &= Z(t_n) - S_k Z(t_{n-1}) -
    \Phi(Z(t_{n-1}),t_{n-1},k), \quad n \in \{1,\ldots,N_k\}. 
  \end{split}
  \label{eq3:opRN}
\end{align}
Below we show that the operators $\mathcal{R}_k$ are well-defined under
Assumptions \ref{as:initial} and \ref{as:stab} for all $k \in (0,T]$. Further,
under these conditions it holds that $\mathcal{R}_k[X_k]=  0 \in
\mathcal{G}_p(\mathcal{T}_k)$ for all $k \in (0,T]$, where $X_k \in
\mathcal{G}_p(\mathcal{T}_k)$ is the discrete time stochastic process generated
by the numerical scheme \eqref{eq3:scheme}. The mappings $\mathcal{R}_k$ are
therefore called \emph{residual operators} associated to the numerical scheme
\eqref{eq3:scheme}. 

The following definition contains our notion of stability.

\begin{definition}
  \label{def:stab}
  Let $p \in [2,\infty)$. 
  The numerical scheme \eqref{eq3:scheme} is called \emph{bistable} (with
  respect to the norms $\|\cdot\|_{0,p}$ and $\|\cdot\|_{-1,p}$) if the
  residual operators $\mathcal{R}_k \colon \mathcal{G}_p(\mathcal{T}_k) \to 
  \mathcal{G}_p(\mathcal{T}_k)$ are well-defined and bijective for all $k \in
  (0,T]$ and if there exists a constant
  $C_{\mathrm{Stab}}$ independent of $k \in (0,T]$ such that 
  \begin{align}
    \label{eq3:bistab}
    \frac{1}{C_{\mathrm{Stab}}} \big\| \mathcal{R}_k[Y] -
    \mathcal{R}_k[Z] \big\|_{-1,p} \le 
    \| Y - Z \|_{0,p} \le C_{\mathrm{Stab}} \big\| \mathcal{R}_k[Y] -
    \mathcal{R}_k[Z] \big\|_{-1,p}
  \end{align}
  for all $k \in (0,T]$ and $Y, Z \in \mathcal{G}_p(\mathcal{T}_k)$. 
\end{definition}

Therefore, for a bistable numerical scheme, the distance between two arbitrary
adapted grid functions can be estimated by the distance of their residuals
measured with respect to the stochastic Spijker norm and vice versa. In Section
\ref{sec:stab} we show that Assumptions \ref{as:initial} to \ref{as:stab} are
sufficient conditions for the stability of the numerical scheme
\eqref{eq3:scheme}. 

The counterpart of the notion of stability is the so-called \emph{consistency}
of the numerical scheme, which we define in the same way as in 
\cite{beynkruse2010,kruse2011}. For this we denote by
$Z|_{\mathcal{T}_k} \in \mathcal{G}_p(\mathcal{T}_k)$ the \emph{restriction} 
of a $p$-fold integrable, adapted and continuous stochastic process $Z \colon
[0,T] \times \Omega \to H$ to the set $\mathcal{G}_p(\mathcal{T}_k)$, that is
\begin{align*}
    Z|_{\mathcal{T}_k}(t_n) := Z(t_n), \quad n \in
    \{0,\ldots,N_k\}.
\end{align*}

\begin{definition}
  \label{def:lvlcons}
  Let $p \in [2, \infty)$. We say that the numerical scheme \eqref{eq3:scheme}
  is \emph{consistent} of order $\gamma > 0$ with respect to the SPDE
  \eqref{eq1:SPDE} if
  there exists a constant $C_\mathrm{Cons}$ independent of $k \in (0,T]$ such
  that 
  \begin{align*}
    \big\| \mathcal{R}_k [ X|_{\mathcal{T}_k} ] \big\|_{-1,p} \le
    C_{\mathrm{Cons}} k^{\gamma} 
  \end{align*}
  for all $k \in (0,T]$, where $X$ is the mild solution to \eqref{eq1:SPDE}.
\end{definition}

The term $\| \mathcal{R}_k [ X|_{\mathcal{T}_k} ] \|_{-1,p}$ is called
\emph{local truncation error} or \emph{consistency error}.
Finally, we introduce the notion of strong convergence. 

\begin{definition}
  \label{def:conv}
  Let $p \in [2, \infty)$. We say that the numerical scheme \eqref{eq3:scheme} 
  is \emph{strongly convergent} of order $\gamma > 0$, if there exists a
  constant $C$ independent of $k \in (0,T]$ such that
  \begin{align*}
    \big\| X_k - X|_{\mathcal{T}_k} \big\|_{0,p} \le C k^{\gamma} 
  \end{align*}
  for all $k \in (0,T]$, where $X_k \in \mathcal{G}_p(\mathcal{T}_k)$, $k \in
  (0,T]$ are the grid functions generated by the numerical scheme
  \eqref{eq3:scheme} and $X$ denotes the mild solution to \eqref{eq1:SPDE}. 
\end{definition}

\begin{theorem}
  \label{th:lvlconv}
  A bistable numerical scheme of the form \eqref{eq3:scheme} is strongly
  convergent of order $\gamma > 0$ if and only if it is consistent of order
  $\gamma > 0$. In particular, it holds
  \begin{align*}
    \frac{1}{C_{\mathrm{Stab}}} \big\| \mathcal{R}_k[X|_{\mathcal{T}_k} ]
    \big\|_{-1,p} \le \big\| X_k - X|_{\mathcal{T}_k} \big\|_{0,p} 
    \le C_{\mathrm{Stab}} \big\| \mathcal{R}_k[X|_{\mathcal{T}_k} ]
    \big\|_{-1,p}  
  \end{align*}
  for all $k \in (0,T]$, where $X_k \in \mathcal{G}_p(\mathcal{T}_k)$, $k \in
  (0,T]$, denotes the family of grid functions generated by the numerical
  scheme \eqref{eq3:scheme} and $X$ is the mild solution to \eqref{eq1:SPDE}. 
\end{theorem}

\begin{proof}
  First, let us recall that the residual operators $\mathcal{R}_k \colon
  \mathcal{G}_p(\mathcal{T}_k) \to \mathcal{G}_p(\mathcal{T}_k)$ satisfy
  $\mathcal{R}_k [ X_k ] = 0$ for every $k \in (0,T]$. Thus,
  by the bistability of the numerical scheme and we obtain 
  \begin{align*}
    \frac{1}{C_{\mathrm{Stab}}} \big\|
    \mathcal{R}_k\big[X|_{\mathcal{T}_k} \big] \big\|_{-1,p}
    \le \big\| X_k - X|_{\mathcal{T}_k} \big\|_{0,p} \le  
    C_{\mathrm{Stab}} \big\|
    \mathcal{R}_k\big[X|_{\mathcal{T}_k} \big] \big\|_{-1,p}.  
  \end{align*}
  Consequently, the assertion follows directly from the definitions of
  consistency and strong convergence.  
\end{proof}

\subsection{Assumptions on the numerical scheme}
\label{subsec:Assumptions}

In this subsection some assumptions on the abstract numerical
scheme \eqref{eq3:scheme} are collected, which assure its stability as we will
show in Section~\ref{sec:stab}.  

\begin{assumption}[Initial value]
  \label{as:initial}
  Let $p \in [2,\infty)$. 
  The initial condition $\xi \colon \Omega \to H$ is a
  $p$-fold integrable and $\mathcal{F}_0/\mathcal{B}(H)$-measurable random
  variable. 
\end{assumption}

The next two assumptions are concerned with the family of linear operators 
$S_k$, $k \in (0,T]$, and the increment function $\Phi$. 

\begin{assumption}[Linear stability]
  \label{as:linstab}
  For the family of bounded linear operators $S_{k} \colon H \to H$, $k \in
  (0,T]$, there exists a constant $C_S$ independent of $k \in (0,T]$ such that 
  \begin{align*}
      \sup_{k \in (0,T]} \sup_{n \in \{1,\ldots,N_k\}} \| S_k^n \|_{\LB(H)} \le
      C_S. 
  \end{align*}  
\end{assumption}

\begin{assumption}[Nonlinear stability]
  \label{as:stab}
  Let $p \in [2,\infty)$ be the same as in Assumption \ref{as:initial}.
  For every  $(t,k) \in \mathbb{T}$ the mapping $\Phi(\cdot,t,k) \colon H
  \times \Omega \to H$ is measurable with respect to $\B(H) \otimes \F_{t +
  k}/\B(H)$. Further, there exists a constant $C_\Phi$ such that
  \begin{align}
    \label{eq3:stab1}
    &\Big\| \sum_{j = m}^{n} S^{n-j}_k \Phi(0,t_{j-1},k) \Big\|_{L_p(\Omega;H)}
    \le C_\Phi  \big(t_n - t_{m-1} \big)^{\frac{1}{2}}
  \end{align}
  for all $k \in (0,T]$ and $n,m \in \{1, \ldots, N_k\}$ with $n \ge m$.
  In addition, it holds
  \begin{align}
    \label{eq3:stab2}
    \begin{split}
      &\Big\| \sum_{j = 1}^{n} S_k^{n-j} \big(
      \Phi( Y(t_{j-1}),t_{j-1}, k ) 
      - \Phi(Z(t_{j-1}),t_{j-1},k) \big) \Big\|_{L_p(\Omega;H)}^2\\
      &\quad \le C_\Phi^2 k \sum_{j = 1}^{n} 
      \big( t_n - t_{j-1} \big)^{-\frac{1}{2}}
      \big\| Y(t_{j-1}) - Z(t_{j-1}) \big\|_{L_p(\Omega;H)}^2  
    \end{split}
  \end{align}
  for all $k \in (0,T]$, $n \in \{1,\ldots,N_k\}$ and all $Y, Z \in
  \mathcal{G}_p(\mathcal{T}_k)$.
\end{assumption}

Let us remark that from Assumption \ref{as:stab} it follows directly that
\begin{align}
    \label{eq3:stab1b}
    \big\| \Phi( Z(t_{n-1}) ,t_{n-1}, k)\big\|_{L_p(\Omega;H)} \le C_\Phi
    k^{\frac{1}{4}} \big(T^{\frac{1}{4}} + \| Z(t_{n-1})
    \|_{L_p(\Omega;H)} \big)  
\end{align}
for all $k \in (0,T]$, $n \in \{1,\ldots,N_k\}$ and all $Z \in
\mathcal{G}_p(\mathcal{T}_k)$. Indeed, fix $k \in (0,T]$ and $n 
\in \{1,\ldots,N_k\}$ and define $\hat{Z} \in \mathcal{G}_p(\mathcal{T}_k)$ by
\begin{align*}
  \hat{Z}(t_j) = 
  \begin{cases}
    Z(t_{n-1}),& j = n-1,\\
    0,& j \neq n-1.
  \end{cases}  
\end{align*}
Then, we get from \eqref{eq3:stab2} 
\begin{align*}
  & \big\| \Phi\big( Z(t_{n-1}), t_{n-1}, k \big) \big\|_{L_p(\Omega;H)} \\
  &\quad = \Big\| \sum_{j = 1}^{n} \Big[ 
  S^{n-j}_k \big( \Phi\big(\hat{Z}(t_{j-1}),t_{j-1},k\big) - 
  \Phi\big(0, t_{j-1},k \big) \big) \Big] 
  +  \Phi\big( 0 , t_{n-1}, k\big) \Big\|_{L_p(\Omega;H)}\\ 
  &\quad \le C_\Phi k^{\frac{1}{4}} \big\| Z(t_{n-1})
    \big\|_{L_p(\Omega;H)} + \big\| \Phi\big( 0, t_{n-1}, k \big)
    \big\|_{L_p(\Omega;H)}. 
\end{align*}
Further, \eqref{eq3:stab1} applied with $n= m$ yields
\begin{align*}
  \big\| \Phi\big( 0, t_{n-1}, k \big)
  \big\|_{L_p(\Omega;H)} \le C_\Phi T^{\frac{1}{4}} k^{\frac{1}{4}}  
\end{align*}
which completes the proof of \eqref{eq3:stab1b}.

\subsection{Bistability of the numerical scheme}
\label{sec:stab}

In this subsection we demonstrate that Assumptions \ref{as:initial} to
\ref{as:stab} are sufficient for the bistability of the numerical scheme.

\begin{theorem}
  Let Assumptions \ref{as:initial} to \ref{as:stab} be satisfied with $p \in
  [2,\infty)$. Then, the mappings 
  $\mathcal{R}_k\colon \mathcal{G}_p(\mathcal{T}_k) \to
  \mathcal{G}_p(\mathcal{T}_k)$ are 
  well-defined and bijective for all $k \in (0,T]$. Further, the numerical
  scheme \eqref{eq3:scheme} is bistable.  
  \label{th:stab}
\end{theorem}

\begin{proof}
  Let $k \in (0,T]$ be arbitrary. We first prove that $\mathcal{R}_k \colon
  \mathcal{G}_p(\mathcal{T}_k) \to \mathcal{G}_p(\mathcal{T}_k)$ is indeed
  well-defined. For all $n \in \{0,\ldots,N_k\}$ and $Z \in
  \mathcal{G}_p(\mathcal{T}_k)$ the random variable $\mathcal{R}_k[Z](t_n)$ is
  $\mathcal{F}_{t_n}$-measurable. In addition, by Assumptions
  \ref{as:initial} and \ref{as:linstab}
  and \eqref{eq3:stab1b} it follows that the $\mathcal{R}_k[Z](t_n)$ is
  also $p$-fold integrable for all $n \in \{0,\ldots,N_k\}$. Therefore, it
  holds $\mathcal{R}_k[Z] \in \mathcal{G}_p(\mathcal{T}_k)$. 

  Given $Y,Z \in \mathcal{G}_p(\mathcal{T}_k)$ with $\mathcal{R}_k[Y] =
  \mathcal{R}_k[Z]$, then it particularly holds $\mathcal{R}_k[Y](t_0) =
  \mathcal{R}_k[Z](t_0)$ from which we deduce $Y(t_0) = Z(t_0)$. Further, under
  the assumption that for some $n \in \{0,\ldots,N_k-1\}$, we have shown that
  $Y(t_j) = Z(t_j)$ for all $j \in \{0,\ldots,n\}$ then it follows by
  \eqref{eq3:opRN} 
  \begin{align*}
    0 = \mathcal{R}_k[Y]( t_{n+1}) - \mathcal{R}_k[Z](t_{n+1}) = Y(t_{n+1}) -
    Z(t_{n+1}). 
  \end{align*}
  Hence, $Y(t_{n+1}) = Z(t_{n+1})$ which proves that $\mathcal{R}_k$ is
  injective. 

  Further, for arbitrary $V \in \mathcal{G}_p(\mathcal{T}_k)$ the grid 
  function $Z \in \mathcal{G}_p(\mathcal{T}_k)$ defined by 
  \begin{align}
    \label{eq3:disc_var_const}
    \begin{split}
    Z(t_0) &:= V(t_0) + \xi,\\
    Z(t_n) &:= S_k^n Z(t_0) + \sum_{j = 1}^n S_k^{n-j} \big(
    \Phi(Z(t_{j-1}), t_{j-1}, k) + V(t_{j}) \big),
  \end{split}
  \end{align}
  for all $n \in \{1,\ldots,N_k\}$, satisfies $R^{N}[Z] = V$, as one directly
  verifies by an inductive argument. Consequently, $R^{N}$ is also surjective.
  In particular, for all $Z \in \mathcal{G}_p(\mathcal{T}_k)$ we equivalently
  rewrite the discrete variation of constants formula
  \eqref{eq3:disc_var_const} as 
  \begin{align}
    \label{eq3:disc_var_const2}
    \begin{split}
      Z(t_0) &= \mathcal{R}_k[Z](t_0) + \xi,\\
      Z(t_n) &= S_k^n Z(t_0) + \sum_{j = 1}^n S_k^{n-j} \big(
      \Phi(Z(t_{j-1}),t_{j-1}, k ) + \mathcal{R}_k[Z](t_j) \big)
    \end{split}
  \end{align}
  for all $n \in \{1,\ldots,N_k\}$. Thus, from Assumption \ref{as:linstab} and
  \eqref{eq3:stab2} we obtain 
  \begin{align*}
    \begin{split}
      &\|Y(t_n) - Z(t_n) \|_{L_p(\Omega;H)} \le \big\| S_k^n 
      ( Y(t_0) - Z(t_0) ) \big\|_{L_p(\Omega;H)} \\
      &\qquad + \Big\| \sum_{j = 1}^n S_k^{n-j} \big(
      \Phi(Y(t_{j-1}), t_{j-1}, k) ) - \Phi(Z(t_{j-1}), t_{j-1}, k)
      \big) \Big\|_{L_p(\Omega;H)}\\ 
      &\qquad + \Big\| \sum_{j = 1}^n S_k^{n-j} \big( \mathcal{R}_k[Y](t_j) 
      -  \mathcal{R}_k[Z](t_j) \big) \Big\|_{L_p(\Omega;H)}\\ 
      &\quad\le C_S \big\| Y(t_0) - Z(t_0) \big\|_{L_p(\Omega;H)} +
      \Big\| \sum_{j = 1}^n S_k^{n-j} \big( \mathcal{R}_k[Y](t_j) 
      -  \mathcal{R}_k[Z](t_j) \big) \Big\|_{L_p(\Omega;H)} \\ 
      &\qquad + C_\Phi \Big(
      k \sum_{j = 1}^{n} \big( t_n - t_{j-1} \big)^{-\frac{1}{2}}
      \big\| Y(t_{j-1}) - Z(t_{j-1}) \big\|_{L_p(\Omega;H)}^2 
      \Big)^{\frac{1}{2}} 
    \end{split}
  \end{align*}
  for all $n \in \{1,\ldots,N_k\}$ and all $Y, Z \in
  \mathcal{G}_p(\mathcal{T}_k)$. In addition, we have 
  \begin{align*}
    \|Y(t_0) - Z(t_0) \|_{L_p(\Omega;H)} 
    = \|\mathcal{R}_k[Y](t_0) -\mathcal{R}_k[Z](t_0) \|_{L_p(\Omega;H)}.  
  \end{align*}
  Now, from the definition of the 
  norm $\|\cdot\|_{-1,p}$ in \eqref{eq3:norm1} it directly follows that
  \begin{align*}
    &\|Y(t_n) - Z(t_n) \|_{L_p(\Omega;H)}^2 \le 2 (1 + C_S)^2 \big\|
    \mathcal{R}_k[Y] -  \mathcal{R}_k[Z] \big\|_{-1,p}^2 \\ 
    &\qquad \qquad \qquad + 2 C_\Phi^2 k
    \sum_{j = 1}^{n}\big( t_n - t_{j-1} \big)^{-\frac{1}{2}} \big\|
    Y(t_{j-1}) - Z(t_{j-1}) \big\|_{L_p(\Omega;H)}^2    
  \end{align*}
  and an application of the discrete Gronwall lemma (see Lemma
  \ref{lem:gronwall}) completes the proof of the right hand side inequality in
  \eqref{eq3:bistab}.

  Similarly, by rearranging \eqref{eq3:disc_var_const2} and an application of
  \eqref{eq3:stab2} we obtain 
  \begin{align*}
    &\Big\| \sum_{j = 1}^n S_k^{n-j} \big( \mathcal{R}_k[Y](t_j) 
    -  \mathcal{R}_k[Z](t_j) \big) \Big\|_{L_p(\Omega;H)}\\
    &\quad \le  \|Y(t_n) - Z(t_n) \|_{L_p(\Omega;H)} +
    \| S_k^n(Y(t_0) - Z(t_0)) \|_{L_p(\Omega;H)}\\
    &\qquad + \Big\| \sum_{j = 1}^n S_k^{n-j} \big(
      \Phi(Y(t_{j-1}), t_{j-1}, k) ) - \Phi(Z(t_{j-1}), t_{j-1}, k)
      \big) \Big\|_{L_p(\Omega;H)}\\
      &\quad \le 2 \|Y - Z \|_{0,p} + C_\Phi \Big(
      k \sum_{j = 1}^{n} \big( t_n - t_{j-1} \big)^{-\frac{1}{2}}
      \big\| Y(t_{j-1}) - Z(t_{j-1}) \big\|_{L_p(\Omega;H)}^2 
      \Big)^{\frac{1}{2}}\\
      &\quad \le \Big( 2 + C_\Phi \Big(k \sum_{j = 1}^{n} \big( t_n - t_{j-1}
      \big)^{-\frac{1}{2}} \Big)^{\frac{1}{2}} \Big) \|Y - Z \|_{0,p}
  \end{align*}
  for all $n \in \{1,\ldots,N_k\}$ and $Y, Z \in \mathcal{G}_p(\mathcal{T}_k)$.
  Since we have
  \begin{align}
    \label{eq4:sumsing}
    k \sum_{j = 1}^n \big(t_n - t_{j-1}\big)^{-\frac{1}{2}} \le
    \int_{0}^{t_n} \sigma^{-\frac{1}{2}} \diff{\sigma} \le
    2 t_n^{\frac{1}{2}} \le 2 T^{\frac{1}{2}},
  \end{align}
  we have also shown the validity of the inequality on the left-hand side of
  \eqref{eq3:bistab}.
\end{proof}

A proof of the following version of Gronwall's lemma is given in
\cite[Lemma~7.1]{elliott1992}.

\begin{lemma}[Discrete Gronwall lemma]
  Let $T > 0$, $k \in (0,T]$, $\eta \in (0,1]$ and a real-valued nonnegative 
  sequence $x_n$, $n \in \{0,\ldots,N_k\}$,
  be given. If there exist constants $C_1, C_2 \ge 0$ such that
  \begin{align*}
    x_n \le C_1 + C_2 k \sum_{j = 1}^{n} \big( t_n - t_{j-1} \big)^{-1+\eta}
    x_{j-1} \quad \text{ for all } n = 0, \ldots, N_k.
  \end{align*}
  Then, there exists a constant $C = C(C_2, T, \eta)$, independent of $k$,
  such that
  \begin{align*}
    x_n \le C C_1 \quad \text{ for all } n = 0, \ldots, N_k.
  \end{align*}  
  \label{lem:gronwall}
\end{lemma}

Having this established we directly deduce the following norm estimate 
for the numerical scheme \eqref{eq3:scheme}.

\begin{corollary}
  \label{cor:stab}
  For $k \in (0,T]$ let $X_k \in \mathcal{G}_p(\mathcal{T}_k)$ be the grid
  function, which is generated by the numerical scheme \eqref{eq3:scheme}. Under
  Assumptions \ref{as:initial} to \ref{as:stab} with $p \in [2,\infty)$ it
  holds that  
  \begin{align*}
    \big\| X_k \big\|_{0,p} \le C_{\mathrm{Stab}} \big( \big\| \xi
    \big\|_{L_p(\Omega;H)} + C_\Phi T^{\frac{1}{2}} \big), 
  \end{align*}
  for all $k \in (0,T]$. 
\end{corollary}

\begin{proof}
  Under the given assumptions the numerical scheme \eqref{eq3:scheme} is stable.
  Since $\mathcal{R}_k[X_k] = 0 \in \mathcal{G}_p(\mathcal{T}_k)$ it holds
  \begin{align*}
    \big\| X_k \big\|_{0,p} = \big\| X_k - 0 \big\|_{0,p} &\le
    C_{\mathrm{Stab}} \big\| \mathcal{R}_k[X_k]-\mathcal{R}_k[0]\big\|_{-1,p}
    \\&= C_{\mathrm{Stab}} \big\| \mathcal{R}_k[0]\big\|_{-1,p}.
  \end{align*}
  Further, from \eqref{eq3:stab1} it follows that 
  \begin{align*}
    \big\| \mathcal{R}_k[0]\big\|_{-1,p} &= \big\| \xi
    \big\|_{L_p(\Omega;H)} + \max_{n \in \{1,\ldots,N_k\}} 
    \Big\| \sum_{j = 1}^n S_k^{n-j} \Phi(0, t_{j-1}, k) 
    \Big\|_{L_p(\Omega;H)}\\ 
    &\le \big\| \xi \big\|_{L_p(\Omega;H)} + C_\Phi T^{\frac{1}{2}},
  \end{align*}
  which completes the proof.
\end{proof}

\subsection{Consistency of the numerical scheme}
\label{sec:lvlcons}

In this section we derive a decomposition of
the local truncation error $\|\mathcal{R}_k [X|_{\mathcal{T}_k} ] \|_{-1,p}$,
which turns out to be useful in the proof of consistency of the Milstein
scheme. In Lemma \ref{lem:cons} it 
is shown that the local truncation error is dominated by a sum of five terms.

The first one is concerned with the distance between the initial conditions of
the SPDE \eqref{eq1:SPDE} and the numerical scheme \eqref{eq3:scheme}. The next
three summands are concerned with the error originating from replacing the
analytic semigroup $S(t)$, $t \in [0,T]$, by the family of bounded linear
operators $S_k$. Finally, the last term deals with the error caused by the
increment function $\Phi$.

\begin{lemma}
  \label{lem:cons}
  Let $X$ be the mild solution to \eqref{eq1:SPDE}. Then the local truncation
  error satisfies the estimate
  \begin{align*}
    &\big\|\mathcal{R}_k [X|_{\mathcal{T}_k} ] \big\|_{-1,p} \le \big\| X(t_0)
    - \xi \big\|_{L_p(\Omega;H)} + \max_{n \in \{1,\ldots,N_k\}} \big\| (S(t_n)
    - S_k^n) X_0 \big\|_{L_p(\Omega;H)} \\ 
    &\quad +\max_{n \in \{1,\ldots,N_k\}} \Big\| \sum_{j = 1}^n
    \int_{t_{j-1}}^{t_{j}} \big( 
    S(t_n - \sigma) - S_k^{n - j + 1} \big) f(X(\sigma)) 
    \diff{\sigma}\Big\|_{L_p(\Omega;H)} \\
    &\quad +\max_{n \in \{1,\ldots,N_k\}} \Big\| \sum_{j = 1}^n
    \int_{t_{j-1}}^{t_{j}} \big( S(t_n - 
    \sigma) - S_k^{n - j + 1} \big) g(X(\sigma)) 
    \diff{W(\sigma)}\Big\|_{L_p(\Omega;H)} \\
    &\quad +\max_{n \in \{1,\ldots,N_k\}} \Big\| \sum_{j = 1}^n S_k^{n-j} \Big(
    - \int_{t_{j-1}}^{t_{j}} S_k f(X(\sigma))
    \diff{\sigma} + \int_{t_{j-1}}^{t_{j}} S_k g(X(\sigma)) \diff{W(\sigma)} \\
    &\qquad \qquad \qquad \qquad   
    - \Phi(X(t_{j-1}),t_{j-1},k) \Big) 
    \Big\|_{L_p(\Omega;H)}    
  \end{align*}
  for all $k \in (0,T]$.
\end{lemma}

\begin{proof}
  The stochastic Spijker norm of $\mathcal{R}_k [X|_{\mathcal{T}_k} ]$ is given
  by 
  \begin{align*}
    &\big\|\mathcal{R}_k [X|_{\mathcal{T}_k} ] \big\|_{-1,p} \\ 
    &\quad = \big\| \mathcal{R}_k [X|_{\mathcal{T}_k} ](t_0)
    \big\|_{L_p(\Omega;H)} + 
    \max_{n \in \{1,\ldots,N_k\}} \Big\| \sum_{j = 1}^n S_k^{n-j} 
    \mathcal{R}_k [X|_{\mathcal{T}_k} ](t_j) \Big\|_{L_p(\Omega;H)}\\
    &\quad = \big\| X(t_0) - \xi \big\|_{L_p(\Omega;H)}\\
    &\qquad +
    \max_{n \in \{1,\ldots,N_k\}} \Big\| \sum_{j = 1}^n S_k^{n-j}
    \big( X(t_j) - S_k X(t_{j-1}) - \Phi(X(t_{j-1}),t_{j-1},k) \big)
    \Big\|_{L_p(\Omega;H)}.
  \end{align*}
  First, we insert the following relationship into the second term
  \begin{align*}
    X(t_j) &= S(t_j - t_{j-1}) X(t_{j-1}) - \int_{t_{j-1}}^{t_{j}} S(t_j -
    \sigma) f(X(\sigma)) \diff{\sigma}\\
    &\qquad+ \int_{t_{j-1}}^{t_{j}} S(t_j - \sigma)
    g(X(\sigma)) \diff{W(\sigma)}, \quad \P\text{-a.s.},
  \end{align*}
  which follows from \eqref{eq1:mild}. Hence, for every $n \in \{1, \ldots,
  N_k\}$ we get
  \begin{align*}
    &\Big\| \sum_{j = 1}^n S_k^{n-j} 
    \big( X(t_j) - S_k X(t_{j-1}) - \Phi(X(t_{j-1}),t_{j-1},k) \big)
    \Big\|_{L_p(\Omega;H)}\\
    &\quad \le \Big\| \sum_{j = 1}^n 
    S_k^{n-j} \Big( \big( S(k) - S_k \big) X(t_{j-1})
    - \int_{t_{j-1}}^{t_{j}} \big( S(t_j - \sigma) - S_k \big) f(X(\sigma))
    \diff{\sigma}\\
    &\qquad \qquad + \int_{t_{j-1}}^{t_{j}} \big( S(t_j - \sigma) - S_k \big)
    g(X(\sigma)) \diff{W(\sigma)} \Big) \Big\|_{L_p(\Omega;H)}\\
    &\qquad + \Big\| \sum_{j = 1}^n S_k^{n-j} \Big(
    - \int_{t_{j-1}}^{t_{j}} S_k f(X(\sigma))
    \diff{\sigma} + \int_{t_{j-1}}^{t_{j}} S_k g(X(\sigma)) \diff{W(\sigma)} \\
    &\qquad \qquad \qquad    
    - \Phi(X(t_{j-1}),t_{j-1},k) \Big) 
    \Big\|_{L_p(\Omega;H)}.
  \end{align*}
  The last summand is already in the desired form. Therefore, it remains to
  estimate the first summand  
  \begin{align*}
    \Theta_n &:= \Big\| \sum_{j = 1}^n 
    S_k^{n-j} \Big( \big( S(k) - S_k \big) X(t_{j-1})
    - \int_{t_{j-1}}^{t_{j}} \big( S(t_j - \sigma) - S_k \big) f(X(\sigma))
    \diff{\sigma}\\
    &\qquad \qquad + \int_{t_{j-1}}^{t_{j}} \big( S(t_j - \sigma) - S_k \big)
    g(X(\sigma)) \diff{W(\sigma)} \Big) \Big\|_{L_p(\Omega;H)}.
  \end{align*}
  For this, we again insert \eqref{eq1:mild} and obtain
  \begin{align*}
     \Theta_n &\le \Big\| \sum_{j = 1}^n
     S_k^{n-j} \big( S(k) - S_k \big) S(t_{j-1}) X_0 \Big\|_{L_p(\Omega;H)} 
     \\ 
     &\quad + \Big\| \sum_{j = 1}^n S_k^{n-j} \Big( \big( S(k) - S_k \big)
     \int_{0}^{t_{j-1}} S(t_{j-1} - \sigma) f(X(\sigma)) \diff{\sigma} \\
     &\qquad \qquad + \int_{t_{j-1}}^{t_{j}} \big( S(t_j - \sigma) - S_k \big)
     f(X(\sigma)) \diff{\sigma} \Big) \Big\|_{L_p(\Omega;H)}\\
     &\quad + \Big\| \sum_{j = 1}^n S_k^{n-j} \Big( \big( S(k) - S_k \big)
     \int_{0}^{t_{j-1}} S(t_{j-1} - \sigma) g(X(\sigma)) \diff{W(\sigma)}\\
     &\qquad \qquad + \int_{t_{j-1}}^{t_{j}} \big( S(t_j - \sigma) - S_k \big) 
     g(X(\sigma)) \diff{W(\sigma)} \Big) \Big\|_{L_p(\Omega;H)} =:
     \Theta_n^1 + \Theta_n^2 + \Theta_n^3.
  \end{align*}
  Next, we apply the fact that
  \begin{align}
    \label{eq:binomial}
    \sum_{j = 1}^n S_k^{n-j} \big( S(k) - S_k \big) 
    S(t_{j-1}) = S(t_n) - S_k^n
  \end{align}
  for all $n \in \{1,\ldots, N_k\}$. This yields for the term $\Theta_n^1$ the
  estimate
  \begin{align}
    \label{eq3:Theta1}
    \Theta_n^1 =
    \Big\| \sum_{j = 1}^n S_k^{n-j} \big( S(k) - S_k \big) 
    S(t_{j-1}) X_0 \Big\|_{L_p(\Omega;H)}
    = \big\| (S(t_n) - S_k^n) X_0 \big\|_{L_p(\Omega;H)}
  \end{align}
  for all $n \in \{1,\ldots, N_k\}$. In addition, it holds
  \begin{align}
    \label{eq3:Theta2}
    \begin{split}
      \Theta_n^2 &= \Big\| \sum_{j = 1}^n \int_{t_{j-1}}^{t_{j}} \big(
      S(t_n - \sigma) - S_k^{n - j + 1} \big) f(X(\sigma)) 
      \diff{\sigma}\Big\|_{L_p(\Omega;H)},
    \end{split}
  \end{align}
  as well as
  \begin{align}
    \label{eq3:Theta3}
    \begin{split}
      \Theta_n^3 &= \Big\| \sum_{j = 1}^n \int_{t_{j-1}}^{t_{j}} \big( S(t_n -
      \sigma) - S_k^{n - j + 1} \big) g(X(\sigma)) 
      \diff{W(\sigma)}\Big\|_{L_p(\Omega;H)}
    \end{split}
  \end{align}
  for all $n \in \{1, \ldots, N_k\}$. Indeed, for a given $\sigma \in
  (0,t_{N_k}]$ let $\ell(\sigma) \in\N$ be determined by $t_{\ell(\sigma) - 1}
  < \sigma \le t_{\ell(\sigma)}$. Then, by interchanging summation and 
  integration we obtain
  \begin{align*}
    & \sum_{j = 1}^n S_k^{n-j} \big( S(k) - S_k \big)
    \int_{0}^{t_{j-1}} S(t_{j-1} - \sigma) f(X(\sigma)) \diff{\sigma}\\
    &\quad = \sum_{j = 1}^n \int_{0}^{t_{n-1}} \one_{[0,t_{j-1}]}(\sigma) 
    S_k^{n-j} \big( S(k) - S_k \big) S(t_{j-1} - \sigma) f(X(\sigma))
    \diff{\sigma}\\
    &\quad = \int_{0}^{t_{n-1}} \sum_{j = \ell(\sigma) +1}^n S_k^{n-j} \big(
    S(k) - S_k \big) S(t_{j-1} - t_{\ell(\sigma)}) S(t_{\ell(\sigma)} - \sigma)
    f(X(\sigma)) \diff{\sigma}\\
    &\quad = \int_{0}^{t_{n-1}} \big( S(t_{n} - t_{\ell(\sigma)}) -
    S_k^{n-\ell(\sigma)} \big) S(t_{\ell(\sigma)} - \sigma)
    f(X(\sigma)) \diff{\sigma}\\
    &\quad = \sum_{j = 1}^{n-1} \int_{t_{j-1}}^{t_{j}} 
    \big( S(t_{n} - t_{j}) -
    S_k^{n-j} \big) S(t_{j} - \sigma) f(X(\sigma)) \diff{\sigma},
  \end{align*}
  where we applied \eqref{eq:binomial} in the fourth step. Therefore, it holds
  \begin{align*}
    &\sum_{j = 1}^n S_k^{n-j} \Big( \big( S(k) - S_k \big)
    \int_{0}^{t_{j-1}} S(t_{j-1} - \sigma) f(X(\sigma)) \diff{\sigma} \\
    &\qquad \qquad + \int_{t_{j-1}}^{t_{j}} \big( S(t_j - \sigma) - S_k \big)
    f(X(\sigma)) \diff{\sigma} \Big)\\
    &\quad = \sum_{j = 1}^{n-1} \int_{t_{j-1}}^{t_{j}} 
    \big( S(t_{n} - t_{j}) - S_k^{n-j} \big) S(t_{j} - \sigma) f(X(\sigma))
    \diff{\sigma} \\
    &\qquad \qquad \qquad + \sum_{j = 1}^n S_k^{n-j} \int_{t_{j-1}}^{t_{j}}
    \big( S(t_j - \sigma) - S_k \big) f(X(\sigma)) \diff{\sigma}\\ 
    &\quad = \sum_{j = 1}^n \int_{t_{j-1}}^{t_{j}} \big( S(t_{n} - \sigma) -
    S_k^{n-j + 1} \big) f(X(\sigma)) \diff{\sigma}.
  \end{align*}
  This completes the proof of \eqref{eq3:Theta2} and the same arguments also
  yield \eqref{eq3:Theta3}.
\end{proof}

\begin{remark}
  In the finite dimensional situation with $H = \R^d$, $U = \R^m$, $d, m \in
  \N$ and $A = 0$ the SPDE \eqref{eq1:SPDE} becomes a 
  stochastic ordinary differential equation (SODE). In this situation we have
  $S(t) = \Id_H$ for all $t \in [0,T]$. If one applies a numerical scheme with
  $S_k = \Id_H$, then the error decomposition in Lemma \ref{lem:cons} actually
  holds with equality and it coincides with the stochastic Spijker norm from 
  \cite{beynkruse2010}. Compare also with \cite{kruse2011}, where the
  application of the maximum occurs inside the expectation.
\end{remark}

\section{Bistability of the Milstein-Galerkin finite element scheme} 
\label{sec:Mil}

In this section we embed the Milstein-Galerkin finite element scheme
\eqref{eq4:Milstein} into the abstract framework of Section
\ref{sec:numscheme}. Then, we prove that Assumptions \ref{as:initial} to
\ref{as:stab} are satisfied and we consequently conclude the bistability of the
scheme. 

For the embedding we first set $\xi_h = P_h X_0$ and
\begin{align*}
  S_{k,h} := \big( \Id_H + k A_h \big)^{-1} P_h \in \LB(H)
\end{align*}
for every $h \in (0,1]$. Let us note that in contrast to Section
\ref{subsec:Galerkin} the operator $S_{k,h}$ includes the orthogonal
projector $P_h$ and is therefore defined as an operator from $H$ to $H$.

Further, the increment function $\Phi_h \colon H
\times \mathbb{T} \times \Omega \to H$, $h \in (0,1]$, is given by
\begin{align}
  \label{eq4:Phih}
  \begin{split}
    \Phi_h(x,t,k) &= - k S_{k,h} f(x) + S_{k,h}g(x) \big(W(t+k) - W(t) \big)\\
    &\qquad + S_{k,h} \int_{t}^{t + k} g'(x)\Big[ \int_{t}^{\sigma_1} g(x)
    \diff{W(\sigma_2)} \Big] \diff{W(\sigma_1)} 
  \end{split}
\end{align}
for all $(t,k) \in \mathbb{T}$ and $x \in H$. 

\begin{theorem}
  \label{th:Milstab}
  Under Assumptions \ref{as4:initial} to \ref{as4:g} the Milstein-Galerkin
  finite element scheme  \eqref{eq4:Milstein} is bistable for every $h \in
  (0,1]$. The stability constant $C_{\mathrm{Stab}}$ can be chosen to be
  independent of $h \in (0,1]$. 
\end{theorem}

\begin{proof}
  First, let $h \in (0,1]$ be an arbitrary but fixed parameter value of the
  spatial discretization. By Theorem \ref{th:stab} it is sufficient to show
  that Assumptions \ref{as:initial} to \ref{as:stab} are satisfied.
  
  Regarding Assumption \ref{as:initial} it directly follows from Assumption
  \ref{as4:initial} that $\xi_h = P_h X_0$ is $p$-fold integrable and
  $\F_0/\mathcal{B}(H)$-measurable. Furthermore, it holds
  \begin{align}
    \label{eq4:stabini}
    \big\| \xi_h \big\|_{L_p(\Omega;H)} \le \big\| X_0 \big\|_{L_p(\Omega;H)},
  \end{align}
  that is, the norm of $\xi_h$ is bounded independently of $h \in (0,1]$ by the
  norm of $X_0$.

  The stability of the family of linear operators $S_{k,h}$, $k \in (0,T]$,
  follows from \eqref{eq2:discsmoothing} with $\rho = 0$ which yields
  \begin{align*}
    \big\| S_{k,h}^n x \big\| = \big\| \big( ( \Id_H + kA_h)^{-1} P_h \big)^n x
    \big\| = \big\| ( \Id_H + kA_h)^{-n} P_h x \big\| \le C \| x \|
  \end{align*}
  for all $x \in H$ and $n \in \{1,\ldots,N_k\}$. Consequently, Assumption
  \ref{as:linstab} 
  is satisfied with $C_S = C$ and the constant is also independent of $h \in
  (0,1]$ and $k \in (0,T]$.

  Hence, it remains to investigate if Assumption \ref{as:stab} is also
  fulfilled.
  First, for every $p \in [2,\infty)$ and for all $m,n \in
  \{1,\ldots,N_k\}$ with $n \ge m$ it holds
  \begin{align*}
    &\Big\| \sum_{j = m}^n S_{k,h}^{n-j} \Phi_h(0,t_{j-1},k)
    \Big\|_{L_p(\Omega;H)} \\ 
    &\quad \le \Big\| \sum_{j = m}^n S_{k,h}^{n-j+1} f(0) k
    \Big\|_{L_p(\Omega;H)} + \Big\| \sum_{j = m}^n S_{k,h}^{n-j+1} g(0)
    \big(W(t_{j}) - W(t_{j-1})\big)\Big\|_{L_p(\Omega;H)}\\
    &\qquad + \Big\| \sum_{j = m}^n S_{k,h}^{n-j+1}  \int_{t_{j-1}}^{t_j}
    g'(0)\Big[ \int_{t_{j-1}}^{\sigma_1} g(0) \diff{W(\sigma_2)} \Big]
    \diff{W(\sigma_1)} \Big\|_{L_p(\Omega;H)}\\
    &\quad=: I_1 + I_2 + I_3.
  \end{align*}
  We deal with the three terms separately. By recalling
  \eqref{eq2:defSkh} and \eqref{eq2:smoothSkh} the deterministic
  term $I_1$ is estimated by
  \begin{align}
    \label{eq4:I1}
    \begin{split}
      I_1 = \Big\| \int_{t_{m-1}}^{t_n} S_{k,h}(t_n - \sigma) P_h f(0)
      \diff{\sigma}  \Big\|
      &\le \int_{t_{m-1}}^{t_n} (t_n- \sigma)^{-\frac{1}{2}} \| f(0)
      \|_{-1} \diff{\sigma}\\
      &= 2 (t_n- t_{m-1})^{\frac{1}{2}} \| f(0)\|_{-1}.
    \end{split}
  \end{align}
  For the estimate of $I_2$ we first write the sum as a stochastic integral by
  inserting \eqref{eq2:defSkh}, then we apply Proposition \ref{prop:stoch_int}
  and \eqref{eq2:stabSkh} and obtain
  \begin{align}
    \label{eq4:I2}
    \begin{split}
      I_2 &= \Big\| \int_{t_{m-1}}^{t_n} S_{k,h}(t_n - \sigma) P_h g(0)
      \diff{W(\sigma)}  \Big\|_{L_p(\Omega;H)}\\
      &\le C(p) \Big( \int_{t_{m-1}}^{t_n} \big\| S_{k,h}(t_n -
      \sigma) P_h g(0) \big\|_{\LB_2^0}^2 \diff{\sigma}
      \Big)^{\frac{1}{2}}\\ 
      &\le C \Big( \int_{t_{m-1}}^{t_n} \| g(0) \|_{\LB_2^0}^2 \diff{\sigma}
      \Big)^{\frac{1}{2}}
      = C (t_n- t_{m-1})^{\frac{1}{2}} \| g(0)\|_{\LB_2^0},
    \end{split}
  \end{align}
  where the constant $C$ is again independent of $h \in (0,1]$ and $k \in
  (0,T]$.  

  Before we continue with the estimate of the third term $I_3$, it is
  convenient to introduce the stochastic process $\Gamma_Y \colon [0,T] \times
  \Omega \to H$, which for a given $Y \in \mathcal{G}_p(\mathcal{T}_k)$ is
  defined by
  \begin{align}
    \label{eq4:Gamma}
    \Gamma_Y(\sigma) &:=
    \begin{cases}
      0 \in H, & \text{for } \sigma = 0,\\
      \int^{\sigma}_{t_{j-1}} g(Y(t_{j-1})) \diff{W(\tau)}, & \text{for }
      \sigma \in (t_{j-1}, t_j], \; j \in \{1,\ldots,N_k\}.     
    \end{cases}
  \end{align}
  Note that $\Gamma_Y$ is left-continuous with existing right limits and
  therefore predictable. Further, it holds by Proposition \ref{prop:stoch_int}
  \begin{align*}
    \sup_{\sigma \in [0,T]} \big\| \Gamma_Y(\sigma) \big\|_{L_p(\Omega;\LB_2^0)}
    \le C(p) k^{\frac{1}{2}} \max_{j \in \{1,\ldots,N_k\}} \|
    g(Y(t_{j-1}))\|_{L_p(\Omega;\LB_2^0)} 
  \end{align*}
  for all $p \in [2,\infty)$. 
  
  Together with the same arguments as above and Assumption \ref{as4:g}, this
  yields for $I_3$ that  
  \begin{align}
    \label{eq4:I3}
    \begin{split}
      I_3 &= \Big\| \int_{t_{m-1}}^{t_n} S_{k,h}(t_n - \sigma) P_h g'(0)\big[
      \Gamma_0(\sigma) \big] \diff{W(\sigma)}  \Big\|_{L_p(\Omega;H)}\\
      &\le C(p) \Big( \int_{t_{m-1}}^{t_n} \big\| S_{k,h}(t_n -
      \sigma) P_h g'(0)\big[
      \Gamma_0(\sigma) \big] \big\|_{L_p(\Omega;\LB_2^0)}^2 \diff{\sigma}
      \Big)^{\frac{1}{2}}\\  
      &\le C \Big( \int_{t_{m-1}}^{t_n} \big\|
      g'(0)\big[\Gamma_0(\sigma) \big] 
      \big\|_{L_p(\Omega;\LB_2^0)}^2 \diff{\sigma} \Big)^{\frac{1}{2}}\\ 
      &\le C\, (t_n- t_{m-1})^{\frac{1}{2}} k^{\frac{1}{2}} \| g'(0)
      \|_{\LB(H;\LB_2^0)} \| g(0)\|_{\LB_2^0} \le C C_g^2 T^{\frac{1}{2}} (t_n-
      t_{m-1})^{\frac{1}{2}}. 
    \end{split}
  \end{align}
  Hence, a combination of \eqref{eq4:I1}, \eqref{eq4:I2} and \eqref{eq4:I3}
  completes the proof of \eqref{eq3:stab1}.

  Next, we verify that $\Phi_h$ also satisfies \eqref{eq3:stab2}.
  For every $p \in [2,\infty)$, for all $Y, Z \in \mathcal{G}_p(\mathcal{T}_k)$
  and $n \in \{1,\ldots,N_k\}$ it holds 
  \begin{align*}
    &\Big\| \sum_{j = 1}^n S_{k,h}^{n-j} \big( \Phi_h(Y(t_{j-1}),t_{j-1},k)
    - \Phi_h(Z(t_{j-1}),t_{j-1},k) \big)
    \Big\|_{L_p(\Omega;H)} \\ 
    &\quad \le \Big\| \sum_{j = 1}^n S_{k,h}^{n-j+1} \big( f(Y(t_{j-1})) -
    f(Z(t_{j-1})) \big) k \Big\|_{L_p(\Omega;H)} \\
    &\quad \; + \Big\| \sum_{j = 1}^n S_{k,h}^{n-j+1} \big( g(Y(t_{j-1})) -
    g(Z(t_{j-1})) \big) \big(W(t_{j}) - W(t_{j-1})\big)\Big\|_{L_p(\Omega;H)}\\
    &\quad \; + \Big\| \sum_{j = 1}^n S_{k,h}^{n-j+1}  \int_{t_{j-1}}^{t_j} 
    g'(Y(t_{j-1}))\big[ \Gamma_Y(\sigma) \big]
    - g'(Z(t_{j-1}))\big[ \Gamma_Z(\sigma) \big]  \diff{W(\sigma)}
    \Big\|_{L_p(\Omega;H)}\\
    &\quad =: I_4 + I_5 + I_6.
  \end{align*}
  Again, we bound the three terms separately. For $I_4$ we apply
  \eqref{eq2:discnorm} and \eqref{eq2:discsmoothing} and obtain
  \begin{align*}
    I_4 &\le k \sum_{j = 1}^n \big\|A_h^{\frac{1}{2}} S_{k,h}^{n-j+1} 
    A_h^{-\frac{1}{2}} P_h \big( f(Y(t_{j-1})) - f(Z(t_{j-1})) \big)
    \big\|_{L_p(\Omega;H)}\\  
    &\le k \sum_{j = 1}^n \big(t_n - t_{j-1}\big)^{-\frac{1}{2}}
    \big\| f(Y(t_{j-1})) - f(Z(t_{j-1})) \big\|_{L_p(\Omega;\dot{H}^{-1})}.
  \end{align*}
  Therefore, by an application of Assumption \ref{as4:f} and the Cauchy-Schwarz
  inequality we get 
  \begin{align*}
    I_4^2 &\le C_f^2 k^2 \Big(\sum_{j = 1}^n \big(t_n -
    t_{j-1}\big)^{-\frac{1}{4}} \big(t_n -
    t_{j-1}\big)^{-\frac{1}{4}} \big\| Y(t_{j-1}) - Z(t_{j-1})
    \big\|_{L_p(\Omega;H)} \Big)^2\\
    &\le  C_f^2 k^2 \sum_{j = 1}^n \big(t_n -     t_{j-1}\big)^{-\frac{1}{2}}
    \sum_{j = 1}^n \big(t_n - t_{j-1}\big)^{-\frac{1}{2}} \big\| Y(t_{j-1}) -
    Z(t_{j-1}) \big\|_{L_p(\Omega;H)}^2.
  \end{align*}
  After applying \eqref{eq4:sumsing} the estimate of $I_4^2$ is in the desired
  form of \eqref{eq3:stab2}, that is 
  \begin{align}
    \label{eq4:J1}
    I_4^2 \le 2 C_f^2 T^{\frac{1}{2}} k \sum_{j = 1}^n \big(t_n -
    t_{j-1}\big)^{-\frac{1}{2}} \big\| Y(t_{j-1}) - Z(t_{j-1})
    \big\|_{L_p(\Omega;H)}^2.    
  \end{align}

  In order to apply Proposition \ref{prop:stoch_int} for the remaining two
  terms $I_5$ and $I_6$, we again write the sum as an integral in each term.
  For this we define
  \begin{align*}
    g_Y(\sigma) &:= \one_{(t_{j-1},t_j]}(\sigma) g(Y(t_{j-1})), \\
    g'_Y(\sigma) &:= \one_{(t_{j-1},t_j]}(\sigma) g'(Y(t_{j-1}))\big[
    \Gamma_Y(\sigma) \big]
  \end{align*}
  for all $Y \in \mathcal{G}_p(\mathcal{T}_k)$ and $\sigma \in [0,T]$. Then,
  $I_5$ is estimated by applying Proposition \ref{prop:stoch_int}
  and \eqref{eq2:stabSkh}. Thus we have
  \begin{align*}
      I_5 &= \Big\| \int_{0}^{t_n} S_{k,h}(t_n - \sigma) P_h \big(
      g_Y(\sigma) - g_Z(\sigma) \big) \diff{W(\sigma)}  \Big\|_{L_p(\Omega;H)}\\
      &\le C(p) \Big( \int_{0}^{t_n} \big\| S_{k,h}(t_n - \sigma) P_h \big(
      g_Y(\sigma) - g_Z(\sigma)\big) \big\|_{L_p(\Omega;\LB_2^0)}^2
      \diff{\sigma} \Big)^{\frac{1}{2}} \\ 
      &\le C \Big( \int_{0}^{t_n} \big\| 
      g_Y(\sigma) - g_Z(\sigma) \big\|_{L_p(\Omega;\LB_2^0)}^2 \diff{\sigma}
      \Big)^{\frac{1}{2}}\\
      &= C \Big( k \sum_{j = 1}^n \big\| g(Y(t_{j-1})) - g(Z(t_{j-1})) 
      \big\|_{L_p(\Omega;\LB_2^0)}^2 \Big)^{\frac{1}{2}}. 
  \end{align*}
  Then Assumption \ref{as4:g} yields
  \begin{align}
    \label{eq4:J2}
    \begin{split}
      I_5^2 &\le C^2 C_g k \sum_{j = 1}^n \big\| Y(t_{j-1}) - Z(t_{j-1})
      \big\|_{L_p(\Omega;H)}^2 \\ 
      &\le C^2 T^{\frac{1}{2}} C_g k \sum_{j = 1}^n \big(t_n -
      t_{j-1}\big)^{-\frac{1}{2}} \big\| Y(t_{j-1}) - Z(t_{j-1})
      \big\|_{L_p(\Omega;H)}^2.  
    \end{split}
  \end{align}
  It remains to prove a similar estimate for $I_6$. As above, Proposition
  \ref{prop:stoch_int} and \eqref{eq2:stabSkh} yield
  \begin{align*}
    I_6 &= \Big\| \int_{0}^{t_n} S_{k,h}(t_n - \sigma) P_h \big(
      g_Y'(\sigma) - g_Z'(\sigma) \big) \diff{W(\sigma)}
      \Big\|_{L_p(\Omega;H)}\\ 
      &\le C \Big( \sum_{j = 1}^n \int_{t_{j-1}}^{t_j} \big\|
      g'(Y(t_{j-1}))[\Gamma_Y(\sigma)] - g'(Z(t_{j-1}))[\Gamma_Z(\sigma)] 
      \big\|_{L_p(\Omega;\LB_2^0)}^2 \diff{\sigma} \Big)^{\frac{1}{2}}.  
  \end{align*}
  Next, since
  \begin{align*}
    g'(Y(t_{j-1}))[\Gamma_Y(\sigma)] = \int_{t_{j-1}}^{\sigma}
    g'(Y(t_{j-1})) g(Y(t_{j-1})) \diff{W(\sigma_2)} \in L_p(\Omega;\LB_2^0), 
  \end{align*}
  we obtain by a further application of Proposition \ref{prop:stoch_int} and
  \eqref{eq4:glip}
  \begin{align*}
    I_6^2 &\le C k^2 \sum_{j = 1}^n \big\| g'(Y(t_{j-1}))g(Y(t_{j-1}))
    - g'(Z(t_{j-1}))g(Z(t_{j-1})) \big\|_{L_p(\Omega;\LB_2(U_0,\LB_2^0))}^2\\ 
    &\le C T^{\frac{3}{2}} C_g k \sum_{j = 1}^n \big(t_n -
      t_{j-1}\big)^{-\frac{1}{2}} \big\| Y(t_{j-1}) - Z(t_{j-1})
      \big\|_{L_p(\Omega;H)}^2. 
  \end{align*}
  Hence, together with \eqref{eq4:J1} and \eqref{eq4:J2} the proof of 
  \eqref{eq3:stab2} is complete, where the constant $C_\Phi$ can also be chosen
  to be independent of $h \in (0,1]$.
  
  Concerning the measurability of $\Phi_h$ it is clear, that for every $(x,t,k)
  \in H \times \mathbb{T}$ we have $\Phi_h(x,t,k) \in
  L_p(\Omega,\F_{t+k},\P;H)$. In addition, the same arguments, which have been
  used for the analysis of the terms $I_4$ to $I_6$,
  yield the continuity of $x \mapsto \Phi_h(x,t,k)$ as a mapping
  from $H$ to $L_p(\Omega;H)$. From this we directly deduce the measurability
  of the mapping $(x, \omega) \mapsto \Phi_h(x,t,k)(\omega)$ with respect to
  $\mathcal{B}(H) \otimes \F_{t + k}/ \mathcal{B}(H)$. 

  Finally, after a short inspection of the proof of Theorem \ref{th:stab} 
  we note the following: Since all constants can be chosen to be
  independent of $h \in (0,1]$, there also exists
  a choice of the stability constant $\mathrm{C}_{\mathrm{Stab}}$ for the
  Milstein-Galerkin finite element scheme \eqref{eq4:Milstein}, which is
  likewise independent of the parameter $h \in (0,1]$. 
\end{proof}

\section{Consistency of the Milstein scheme}
\label{subsec:Milcons}
The aim of this section is to investigate, if the Milstein scheme is
consistent. Our result is summarized in the following theorem. Its
proof is based on the decomposition of the local truncation error given in
Lemma \ref{lem:cons} and is split over a series of lemmas.

\begin{theorem}
  \label{th4:cons}
  Let Assumptions \ref{as:Vh} and \ref{as:Vh2} be satisfied by
  the spatial discretization.  
  If Assumptions \ref{as4:initial} to \ref{as4:g} are fulfilled, then the local
  truncation error of the scheme \eqref{eq4:Milstein} satisfies
  \begin{align*}
    \big\| \mathcal{R}_k \big[ X|_{\mathcal{T}_k} \big] \big\|_{-1,p} \le C
    \big( h^{1+r} + k^{\frac{1+r}{2}} \big)
  \end{align*}
  for all $h \in [0,1)$ and $k \in (0,T]$. In particular, if $h$ and $k$ are
  coupled by $h := c k^{\frac{1}{2}}$ for a positive
  constant $c \in \R$, then the Milstein scheme is consistent of order
  $\frac{1+r}{2}$.
\end{theorem}

\begin{lemma}[Consistency of the initial condition]
  \label{lem4:initial}
  Let Assumption \ref{as4:initial} be satisfied with $r \in [0,1]$. 
  Under Assumption \ref{as:Vh} it holds 
  \begin{align*}
    \| X(0) - \xi_h \|_{L_p(\Omega;H)} \le C h^{1+r}
  \end{align*}
  for $\xi_h = P_h X_0$ and for all $h \in (0,1]$. 
\end{lemma}

\begin{proof}
  By the best approximation property of the orthogonal projector $P_h \colon H
  \to V_h$ and by Assumption \ref{as:Vh} it holds
  \begin{align*}
    \| X(0) - \xi_h \|_{L_p(\Omega;H)} &= \| (\Id_H - P_h) X_0
    \|_{L_p(\Omega;H)} \le \|  (\Id_H - R_h) X_0 \|_{L_p(\Omega;H)} \le C
    h^{1+r} 
  \end{align*}
  for all $h \in (0,1]$.
\end{proof}

The next three lemmas are concerned with the consistency of the family of
linear operators $S_{k,h}$, $k \in (0,T]$, $h \in (0,1]$.  

\begin{lemma}
  \label{lem4:lin1}
  Let Assumption \ref{as4:initial} be satisfied for some $r \in [0,1]$. If the
  spatial discretization satisfies Assumption \ref{as:Vh} it holds
  \begin{align*}
    \max_{n \in \{1,\ldots,N_k\}} \big\| \big( S(t_n) - S_{k,h}^n \big) X_0
    \big\|_{L_p(\Omega;H)} \le C \big( h^{1+r} + k^{\frac{1+r}{2}} \big) 
    \big\| A^{\frac{1+r}{2}} X_0 \big\|_{L_p(\Omega;H)}
  \end{align*}
  for all $h \in (0,1]$ and $k \in (0,T]$.
\end{lemma}

\begin{proof}
  The term on the left hand side of the inequality is the error of the fully
  discrete approximation scheme for the linear Cauchy problem $u_t = Au$ with
  the initial condition being a random variable.
  By Assumption \ref{as4:initial} we have that $X_0(\omega) \in \dot{H}^{1+r}$
  for $\P$-almost all $\omega \in \Omega$. Thus, the error estimate from Lemma
  \ref{lem:Fkh1} \emph{(i)} (or \cite[Theorem 7.8]{thomee2006}) yields 
  \begin{align*}
    \big\| \big( S(t_n) - S_{k,h}^n \big) X_0
    \big\|_{L_p(\Omega;H)} \le 
    C \big( h^{1+r} + k^{\frac{1+r}{2}} \big) \big\| A^{\frac{1+r}{2}} X_0
    \big\|_{L_p(\Omega;H)},
  \end{align*}
  for all $h \in (0,1]$, $k \in (0,T]$, where the constant $C$ is also
  independent of $n \in \{1,\ldots,N_k\}$. 
\end{proof}

\begin{lemma}
  \label{lem4:lin2}
  Let Assumptions \ref{as4:initial} to \ref{as4:g} be satisfied for some $r \in
  [0,1]$. If the spatial discretization satisfies Assumption \ref{as:Vh} it
  holds 
  \begin{align*}
    &\max_{n \in \{1,\ldots,N_k\}} \Big\| \sum_{j = 1}^n
    \int_{t_{j-1}}^{t_j} \big( S(t_n - \sigma) - S_{k,h}^{n-j+1} \big)
    f(X(\sigma)) \diff{\sigma} \Big\|_{L_p(\Omega;H)} 
    \le C \big( h^{1+r} +
    k^{\frac{1+r}{2}} \big)  
  \end{align*}
  for all $h \in (0,1]$ and $k \in (0,T]$.
\end{lemma}

\begin{proof}
  First, by recalling \eqref{eq2:errOp} it is convenient to write
  \begin{align*}
    &\sum_{j = 1}^n \int_{t_{j-1}}^{t_j} \big( S(t_n - \sigma) - S_{k,h}^{n-j+1}
    \big) f(X(\sigma)) \diff{\sigma}
     = \int_{0}^{t_n} F_{k,h}(t_n - \sigma) f(X(\sigma)) \diff{\sigma} 
  \end{align*}
  for all $n \in \{1,\ldots,N_k\}$. Then, it follows
  \begin{align*}
    &\Big\| \sum_{j = 1}^n
    \int_{t_{j-1}}^{t_j} \big( S(t_n - \sigma) - S_{k,h}^{n-j+1} \big)
    f(X(\sigma)) \diff{\sigma} \Big\|_{L_p(\Omega;H)} \\
    &\quad \le \Big\| 
    \int_{0}^{t_n} F_{k,h}(t_n - \sigma) \big(f(X(\sigma)) - f(X(t_n))\big)
    \diff{\sigma} \Big\|_{L_p(\Omega;H)}\\
    &\qquad + \Big\| 
    \int_{0}^{t_n} F_{k,h}(t_n - \sigma) f(X(t_n)) \diff{\sigma}
    \Big\|_{L_p(\Omega;H)} =: J_n^1 + J_n^2
  \end{align*}
  for all $n \in \{1,\ldots,N_k\}$.
  We estimate the two summands separately. For $J_{n}^1$ we apply Lemma
  \ref{lem:Fkh1} \emph{(iii)} with $\rho = 1 - r$ and obtain
  \begin{align*}
    J^1_n &\le \int_{0}^{t_n} \big\|
    F_{k,h}(t_n - \sigma) \big(f(X(\sigma)) - f(X(t_n))\big)
    \big\|_{L_p(\Omega;H)} \diff{\sigma}    \\
    &\le C \big( h^{1+r} + k^{\frac{1+r}{2}} \big) 
    \int_{0}^{t_n} (t_n - \sigma)^{-1}
    \big\| f(X(\sigma)) - f(X(t_n))
    \big\|_{L_p(\Omega;\dot{H}^{-1+r})} \diff{\sigma} \\
    &\le C \big( h^{1+r} + k^{\frac{1+r}{2}} \big) \int_{0}^{t_n} (t_n -
    \sigma)^{-1 + \frac{1}{2}} 
    \diff{\sigma} \le C T^{\frac{1}{2}} \big( h^{1+r} + k^{\frac{1+r}{2}}
    \big),
  \end{align*}
  where we also applied \eqref{eq4:flip} and \eqref{eq2:hoelder}. 
  
  The term $J_n^2$ is estimated by an application of Lemma
  \ref{lem:Fkh2} \emph{(i)} with $\rho = 1 - r$, Assumption \ref{as4:f} and
  \eqref{eq2:reg}, which yield
  \begin{align*}
    J_n^2 &\le C  \big( h^{1+r} + k^{\frac{1+r}{2}} \big) \big\| f(X(t_n))
    \big\|_{L_p(\Omega;\dot{H}^{-1 + r})}\\
    &\le  C  \big( h^{1+r} +
    k^{\frac{1+r}{2}} \big) \Big( 1 + \sup_{\sigma \in [0,T]} \| X(\sigma)
    \|_{L_p(\Omega;H)} \Big),
  \end{align*}
  for all  $h \in (0,1]$, $k \in (0,T]$ and $n \in \{1, \ldots, N_k\}$.
  This completes the proof of Lemma \ref{lem4:lin2}.
\end{proof}

\begin{lemma}
  \label{lem4:lin3}
  Let Assumptions \ref{as4:initial} to \ref{as4:g} be satisfied for some $r \in
  [0,1)$. If the spatial discretization satisfies Assumptions \ref{as:Vh} and 
  \ref{as:Vh2} it holds 
  \begin{align*}
    &\max_{n \in \{1,\ldots,N_k\}} \Big\| \sum_{j = 1}^n
    \int_{t_{j-1}}^{t_j} \big( S(t_n - \sigma) - S_{k,h}^{n-j+1} \big)
    g(X(\sigma)) \diff{W(\sigma)} \Big\|_{L_p(\Omega;H)}\\ 
    & \qquad \le C \big( h^{1+r} +
    k^{\frac{1+r}{2}} \big)  
  \end{align*}
  for all $h \in (0,1]$ and $k \in (0,T]$.
\end{lemma}

\begin{proof}
  As in the proof of Lemma \ref{lem4:lin2}, by \eqref{eq2:errOp}, we first
  rewrite the sum inside the norm as
  \begin{align*}
    &\sum_{j = 1}^n \int_{t_{j-1}}^{t_j} \big( S(t_n - \sigma) - S_{k,h}^{n-j+1}
    \big) g(X(\sigma)) \diff{W(\sigma)}
     = \int_{0}^{t_n} F_{k,h}(t_n - \sigma) g(X(\sigma)) \diff{W(\sigma)} 
  \end{align*}
  for all $n \in \{1,\ldots,N_k\}$. Then, it follows by Proposition
  \ref{prop:stoch_int} 
  \begin{align*}
    &\Big\| \sum_{j = 1}^n
    \int_{t_{j-1}}^{t_j} \big( S(t_n - \sigma) - S_{k,h}^{n-j+1} \big)
    g(X(\sigma)) \diff{W(\sigma)} \Big\|_{L_p(\Omega;H)} \\
    &\quad \le C(p) \Big( \E \Big[ \Big(
    \int_{0}^{t_n} \big\| F_{k,h}(t_n - \sigma) g(X(\sigma))
    \big\|^2_{\LB_2^0} \diff{\sigma} \Big)^{\frac{p}{2}}
    \Big] \Big)^{\frac{1}{p}}
    \\
    &\quad \le C(p) \Big( \E \Big[ \Big(
    \int_{0}^{t_n} \big\| F_{k,h}(t_n - \sigma) \big( 
    g(X(\sigma)) - g(X(t_n)) \big) \big\|^2_{\LB_2^0} \diff{\sigma}
    \Big)^{\frac{p}{2}} \Big] \Big)^{\frac{1}{p}}  \\ 
    &\qquad + C(p)\Big( \E \Big[  \Big(
    \int_{0}^{t_n} \big\| F_{k,h}(t_n - \sigma) g(X(t_n))
    \big\|^2_{\LB_2^0} \diff{\sigma} \Big)^{\frac{p}{2}} \Big]
    \Big)^{\frac{1}{p}} 
    =: C(p) \big( J_n^3 + J_n^4 \big)
  \end{align*}
  for all $n \in \{1,\ldots,N_k\}$.
  We estimate the two summands separately. For $J_{n}^3$ we apply Lemma
  \ref{lem:Fkh1} \emph{(i)} with $\mu = 1 + r$ and $\nu = 0$ and obtain
  by \eqref{eq4:glip} and \eqref{eq2:hoelder}
  \begin{align*}
    J^3_n &\le C \big( h^{1+r} + k^{\frac{1+r}{2}}
    \big) \Big( \int_{0}^{t_n} (t_n - \sigma)^{-1-r}  \big\| g(X(\sigma)) -
    g(X(t_n))  \big\|_{L_p(\Omega;\LB_2^0)}^2 \diff{\sigma}
    \Big)^{\frac{1}{2}} \\ 
    &\le C \big( h^{1+r} + k^{\frac{1+r}{2}} \big) \Big( \int_{0}^{t_n} (t_n -
    \sigma)^{- r} \diff{\sigma} \Big)^{\frac{1}{2}} \le C T^{\frac{1-r}{2}}
    \big( h^{1+r} + k^{\frac{1+r}{2}} \big), 
  \end{align*}
  where we also applied the same technique as in the proof of the second
  inequality of Proposition \ref{prop:stoch_int}. 
  
  For the estimate of $J_n^4$ we first apply Lemma
  \ref{lem:Fkh2} \emph{(ii)} with $\rho = r$. Then \eqref{eq4:glin} and 
  \eqref{eq2:reg} yield
  \begin{align*}
    J_n^4 &\le C  \big( h^{1+r} + k^{\frac{1+r}{2}} \big) \big\| g(X(t_n))
    \big\|_{L_p(\Omega;\LB_{2,r}^0)}\\
    &\le  C  \big( h^{1+r} +
    k^{\frac{1+r}{2}} \big) \Big( 1 + \sup_{\sigma \in [0,T]} \| X(\sigma)
    \|_{L_p(\Omega;\dot{H}^r)} \Big),
  \end{align*}
  for all  $h \in (0,1]$, $k \in (0,T]$ and $n \in \{1, \ldots, N_k\}$.
  The proof is complete.
\end{proof}

\begin{remark}
  Let us stress that the case $r=1$ is not included in Lemma \ref{lem4:lin3}.
  The reason for this is found in the estimate of the term $J_n^3$, where a
  blow up occurs for $r = 1$. This problem can be avoided under stronger
  assumptions on $g$ as, for example, the existence of an parameter value $s
  \in (0,1]$ such that 
  \begin{align*}
    \| g(x_1) - g(x_2) \|_{\LB_{2,s}^0} \le C_g \| x_1 - x_2 \|_{s}
  \end{align*}
  for all $x_1, x_2 \in \dot{H}^s$. This is often satisfied for linear $g$ as
  shown in \cite{barth2011}. 
\end{remark}

By Lemma \ref{lem:cons} it therefore remains to investigate the
order of convergence of the fifth and final term, which after inserting
\eqref{eq4:Phih} is dominated by the following two summands 
\begin{align}
  \label{eq4:cons5}
  \begin{split}
  &\max_{n \in \{1,\ldots,N_k\}} \Big\| \sum_{j = 1}^{n} S_{k,h}^{n-j} \Big( -
  \int_{t_{j-1}}^{t_j} S_k f(X(\sigma)) \diff{\sigma} + \int_{t_{j-1}}^{t_j}S_k
  g(X(\sigma)) \diff{W(\sigma)} \\
  &\qquad \qquad \qquad - \Phi_h(X(t_{j-1}),t_{j-1},k) \Big)
  \Big\|_{L_p(\Omega;H)}\\
  &\quad \le \max_{n \in \{1,\ldots,N_k\}} \Big\| \sum_{j = 1}^{n}
  S_{k,h}^{n-j +1 } \int_{t_{j-1}}^{t_j} f(X(\sigma)) - f(X(t_{j-1}))
  \diff{\sigma} \Big\|_{L_p(\Omega;H)}\\
  &\qquad + \max_{n \in \{1,\ldots,N_k\}} \Big\| \sum_{j = 1}^{n}
  S_{k,h}^{n-j +1 } \int_{t_{j-1}}^{t_j} \Big( g(X(\sigma)) - g(X(t_{j-1}))\\
  &\qquad \qquad \qquad \qquad \qquad - g'(X(t_{j-1}))\big[
  \Gamma_{X|_{\mathcal{T}_k}}(\sigma) \big] \Big) \diff{W(\sigma)}
  \Big\|_{L_p(\Omega;H)}, 
  \end{split}
\end{align}
where we recall from \eqref{eq4:Gamma} that
\begin{align}
    \label{eq4:GammaX}
    \Gamma_{X|_{\mathcal{T}_k}}(\sigma) &:=
    \begin{cases}
      0 \in H, & \text{for } \sigma = 0,\\
      \int^{\sigma}_{t_{j-1}} g(X(t_{j-1})) \diff{W(\tau)}, & \text{for }
      \sigma \in (t_{j-1}, t_j], \; j \in \{1,\ldots,N_k\}.     
    \end{cases}
\end{align}
The remaining two lemmas in this section are concerned with the estimate of the
two summands in \eqref{eq4:cons5}.

\begin{lemma}
  \label{lem4:nonlin1}
  Let Assumptions \ref{as4:initial} to \ref{as4:g} be satisfied for some $r \in
  [0,1)$. If the spatial discretization satisfies Assumption \ref{as:Vh} it
  holds 
  \begin{align*}
    &\max_{n \in \{1,\ldots,N_k\}} \Big\| \sum_{j = 1}^{n}
    S_{k,h}^{n-j +1 } \int_{t_{j-1}}^{t_j} f(X(\sigma)) - f(X(t_{j-1}))
    \diff{\sigma} \Big\|_{L_p(\Omega;H)} \le C k^{\frac{1+r}{2}}
  \end{align*}
  for all $h \in (0,1]$ and $k \in (0,T]$.
\end{lemma}

\begin{proof}
  For every $n \in \{1,\ldots,N_k\}$ we first insert the conditional
  expectation in the following way
  \begin{align}
    \label{eq4:consT1}
    \begin{split}
      &\Big\| \sum_{j = 1}^{n} S_{k,h}^{n-j +1 } \int_{t_{j-1}}^{t_j}
      f(X(\sigma)) - f(X(t_{j-1})) \diff{\sigma} \Big\|_{L_p(\Omega;H)}  \\
      &\quad \le \Big\| \sum_{j = 1}^{n} S_{k,h}^{n-j +1 } \int_{t_{j-1}}^{t_j}
      f(X(\sigma)) - \E\big[ f(X(\sigma)) \big| \F_{t_{j-1}} \big]
      \diff{\sigma} \Big\|_{L_p(\Omega;H)}\\
      &\qquad + \Big\| \sum_{j = 1}^{n} S_{k,h}^{n-j +1 } \int_{t_{j-1}}^{t_j}
      \E\big[ f(X(\sigma)) \big| \F_{t_{j-1}} \big]
      - f(X(t_{j-1})) \diff{\sigma} \Big\|_{L_p(\Omega;H)}.
    \end{split}
  \end{align}
  Thus, for the summands of the first term it follows
  \begin{align*}
    \E \Big[ S_{k,h}^{n-j +1 } \int_{t_{j-1}}^{t_j}
    f(X(\sigma)) - \E\big[ f(X(\sigma)) \big| \F_{t_{j-1}} \big]
    \diff{\sigma} \Big| \F_{t_{\ell}} \Big] = 0 \in H    
  \end{align*}
  for every $j,\ell \in \{1, \ldots, n\}$ with $\ell < j$. Consequently, by
  setting 
  \begin{align*}
    M_i := \sum_{j = 1}^i S_{k,h}^{n-j +1 } \int_{t_{j-1}}^{t_j}
      f(X(\sigma)) - \E\big[ f(X(\sigma)) \big| \F_{t_{j-1}} \big]
      \diff{\sigma}, \quad i \in \{0,1,\ldots,n\},
  \end{align*}
  we obtain a discrete time martingale in $L_p(\Omega;H)$. Thus, Burkholder's
  inequality \cite[Th.~3.3]{burkholder1991} is applicable and yields together
  with \eqref{eq2:discsmoothing} and \eqref{eq2:discnorm}
  \begin{align*}
    &\Big\| \sum_{j = 1}^{n} S_{k,h}^{n-j +1 } \int_{t_{j-1}}^{t_j}
    f(X(\sigma)) - \E\big[ f(X(\sigma)) \big| \F_{t_{j-1}} \big]
    \diff{\sigma} \Big\|_{L_p(\Omega;H)}\\
    &\; \le C \Big( \E \Big[ \Big( \sum_{j = 1}^n \Big\| S_{k,h}^{n-j +1 }
    \int_{t_{j-1}}^{t_j} f(X(\sigma)) - \E\big[ f(X(\sigma)) \big| \F_{t_{j-1}}
    \big] \diff{\sigma} \Big\|^2 \Big)^{\frac{p}{2}} \Big]
    \Big)^{\frac{1}{p}} \\
    &\; \le C \Big(  \sum_{j = 1}^n \Big\| A_h^{\frac{1}{2}} S_{k,h}^{n-j +1
    } 
    \int_{t_{j-1}}^{t_j} A_h^{-\frac{1}{2}} P_h \big( f(X(\sigma)) - \E\big[
    f(X(\sigma)) \big| \F_{t_{j-1}} \big] \big) \diff{\sigma}
    \Big\|^2_{L_p(\Omega;H)} \Big)^{\frac{1}{2}}\\
    &\; \le C \Big(  \sum_{j = 1}^n t_{n-j +1}^{-1} k \int_{t_{j-1}}^{t_j} 
    \big\| f(X(\sigma)) - \E\big[ f(X(\sigma)) \big| \F_{t_{j-1}} \big]
    \big\|_{L_p(\Omega;\dot{H}^{-1})}^2 \diff{\sigma} \Big)^{\frac{1}{2}}.
  \end{align*}
  In addition, since $\| \E [ G | \F_{t_{j-1}}]
  \|_{L_p(\Omega;\dot{H}^{-1})} \le  \| G
  \|_{L_p(\Omega;\dot{H}^{-1})}$ for all  $G \in L_p(\Omega;\dot{H}^{-1})$ it
  follows for all $\sigma \in [t_{j-1}, t_j]$
  \begin{align*}
    &\big\| f(X(\sigma)) - \E\big[ f(X(\sigma)) \big| \F_{t_{j-1}} \big]
    \big\|_{L_p(\Omega;\dot{H}^{-1})} \\
    &\quad = \big\| f(X(\sigma)) -
    f(X(t_{j-1})) + 
     \E\big[f(X(t_{j-1})) - f(X(\sigma)) \big| \F_{t_{j-1}} \big]
    \big\|_{L_p(\Omega;\dot{H}^{-1})}\\
    &\quad \le 2 \big\| f(X(\sigma)) -
    f(X(t_{j-1}))\big\|_{L_p(\Omega;\dot{H}^{-1})} \le C | \sigma -
    t_{j-1}|^{\frac{1}{2}},
  \end{align*}
  where we also used \eqref{eq4:flip} and \eqref{eq2:hoelder} in the last step.
  Therefore, in the same way as in \eqref{eq4:sumsing} the estimate of the
  first summand in \eqref{eq4:consT1} is completed by  
  \begin{align*}
    &\Big\| \sum_{j = 1}^{n} S_{k,h}^{n-j +1 } \int_{t_{j-1}}^{t_j}
    f(X(\sigma)) - \E\big[ f(X(\sigma)) \big| \F_{t_{j-1}} \big]
    \diff{\sigma} \Big\|_{L_p(\Omega;H)}\\
    &\quad \le C \Big(  \sum_{j = 1}^n t_{n-j +1}^{-1} k^3 
    \Big)^{\frac{1}{2}}    \le C \Big( k^3 \sum_{j = 1}^n t_{n-j
    +1}^{-r} t_{n-j +1}^{-1+r} \Big)^{\frac{1}{2}}\\
    &\quad \le C \Big( k^3 \sum_{j = 1}^n t_{n-j +1}^{-r} k^{-1+r}
    \Big)^{\frac{1}{2}} \le C k^{\frac{1+r}{2}}. 
  \end{align*}
  For the second summand in \eqref{eq4:consT1} we make use of the mean value
  theorem for Fr\'echet differentiable mappings, which reads
  \begin{align*}
    f(X(\tau_1)) &= f(X(\tau_2)) + \int_{0}^{1} f'(X(\tau_2)+ s (X(\tau_1) -
    X(\tau_2))) \big[ X(\tau_1) - X(\tau_2) \big] \diff{s} 
  \end{align*}
  for all $\tau_1,\tau_2 \in [0,T]$. For convenience we introduce the short
  hand notation
  \begin{align*}
    f'(\tau_1,\tau_2;s) := f'(X(\tau_2)+ s (X(\tau_1) - X(\tau_2)))
  \end{align*}
  for all $\tau_1,\tau_2 \in [0,T]$ and $s \in [0,1]$. Then, by inserting
  \eqref{eq1:mild} we obtain the identity
  \begin{align*}
    &\E\big[ f(X(\sigma)) | \F_{t_{j-1}}\big] - f(X(t_{j-1})) \\
    &\quad = \E \Big[\int_{0}^{1} f'(\sigma,t_{j-1};s)\big[
    \big(S(\sigma-t_{j-1}) - \Id_H\big) X(t_{j-1}) \big] \diff{s} \Big|
    \F_{t_{j-1}} \Big] \\ 
    &\qquad - \E \Big[ \int_{0}^{1}
    f'(\sigma,t_{j-1};s) \Big[  \int_{t_{j-1}}^{\sigma} S(\sigma - \tau)
    f(X(\tau))  \diff{\tau} \Big]\diff{s} \Big| \F_{t_{j-1}} \Big] \\
    &\qquad +  \E \Big[ \int_{0}^{1} f'(\sigma,t_{j-1};s) 
    \Big[  \int_{t_{j-1}}^{\sigma} S(\sigma - \tau)
    g(X(\tau))  \diff{W(\tau)} \Big] \diff{s} \Big| \F_{t_{j-1}} \Big]\\
    &\quad =: \Theta_1 (\sigma,t_{j-1}) + \Theta_2 (\sigma,t_{j-1}) +
    \Theta_3 (\sigma,t_{j-1}), 
  \end{align*}  
  which holds $\P$-almost surely. Hence,   
  the second summand in \eqref{eq4:consT1} satisfies 
  \begin{align}
    \label{eq4:consT2}
    \begin{split}
    &\Big\| \sum_{j = 1}^{n}
    S_{k,h}^{n-j +1 } \int_{t_{j-1}}^{t_j} \E\big[ f(X(\sigma)) |
    \F_{t_{j-1}}\big] - f(X(t_{j-1}))
    \diff{\sigma} \Big\|_{L_p(\Omega;H)}  \\
    &\quad = \Big\| \sum_{j = 1}^{n} S_{k,h}^{n-j +1 } \int_{t_{j-1}}^{t_j}
    \Theta_1 (\sigma,t_{j-1}) + \Theta_2 (\sigma,t_{j-1}) +
    \Theta_3 (\sigma,t_{j-1})
    \diff{\sigma} \Big\|_{L_p(\Omega;H)}\\   
    &\quad \le C \sum_{j = 1}^n t_{n-j+1}^{-\frac{1}{2}} \int_{t_{j-1}}^{t_j}
    \big\| \Theta_1 (\sigma,t_{j-1}) + \Theta_2 (\sigma,t_{j-1}) +
    \Theta_3 (\sigma,t_{j-1})\big\|_{L_p(\Omega;\dot{H}^{-1})} \diff{\sigma}
  \end{split}
  \end{align}
  for every $n \in \{1,\ldots,N_k\}$, where we again applied
  \eqref{eq2:discsmoothing} and \eqref{eq2:discnorm} in the last step.

  Below we show that
  \begin{align}
    \label{eq4:thetai}
    \big\| \Theta_i (\sigma,t_{j-1})\big\|_{L_p(\Omega;\dot{H}^{-1})}  \le C |
    \sigma - t_{j-1} |^{\frac{1+r}{2}}, \quad \text{for } i \in \{1,2,3\}.    
  \end{align}
  Then this is used to complete the estimate of \eqref{eq4:consT2} by
  \begin{align*}
    &\Big\| \sum_{j = 1}^{n}
    S_{k,h}^{n-j +1 } \int_{t_{j-1}}^{t_j} \E\big[ f(X(\sigma)) |
    \F_{t_{j-1}}\big] - f(X(t_{j-1}))
    \diff{\sigma} \Big\|_{L_p(\Omega;H)}  \\
    &\quad \le C \sum_{j = 1}^n t_{n-j+1}^{-\frac{1}{2}} \int_{t_{j-1}}^{t_j}
    | \sigma - t_{j-1}|^{\frac{1+r}{2}} \diff{\sigma} \le C k^{\frac{1+r}{2}},
  \end{align*}
  where we again applied \eqref{eq4:sumsing}. Thus, the assertion is
  proved if we show \eqref{eq4:thetai}. 

  For the estimation of $\Theta_1$ we recall that $\| \E [ G | \F_{t_{j-1}}]
  \|_{L_p(\Omega;\dot{H}^{-1})} \le  \| G
  \|_{L_p(\Omega;\dot{H}^{-1})}$ for all  $G \in L_p(\Omega;\dot{H}^{-1})$ and
  obtain
  \begin{align*}
    &\big\| \Theta_1 (\sigma,t_{j-1})\big\|_{L_p(\Omega;\dot{H}^{-1})}\\
    &\quad \le \Big\| \int_{0}^{1} f'(\sigma,t_{j-1};s)\big[ 
    \big(S(\sigma-t_{j-1}) - \Id_H\big) X(t_{j-1}) \big] \diff{s}
    \Big\|_{L_p(\Omega;\dot{H}^{-1})} \\ 
    &\quad \le\int_{0}^{1} \Big( \E \Big[ \big\|
    f'(\sigma,t_{j-1};s) \big\|_{\LB(H,\dot{H}^{-1})}^p \big\| \big(
    S(\sigma - t_{j-1}) - \Id_H \big) X(t_{j-1}) \big\|^p \Big]
    \Big)^{\frac{1}{p}} \diff{s}\\
    &\quad \le C \sup_{x \in H} \big\| f'(x) \big\|_{\LB(H,\dot{H}^{1+r})} 
    \big\| \big( S(\sigma - t_{j-1}) - \Id_H \big) X(t_{j-1})
    \big\|_{L_p(\Omega;H)}.
  \end{align*}
  Further, from \cite[Ch.~2.6, Th.~6.13]{pazy1983} it follows
  \begin{align*}
    \big\| \big( S(\sigma - t_{j-1}) - \Id_H \big) X(t_{j-1})
    \big\|_{L_p(\Omega;H)} &\le C (\sigma - t_{j-1})^{\frac{1+r}{2}} \|
    X(t_{j-1}) \big\|_{L_p(\Omega;\dot{H}^{1+r})}\\
    &\le C (\sigma - t_{j-1})^{\frac{1+r}{2}} \sup_{\sigma \in [0,T]} \|
    X(\sigma) \|_{L_p(\Omega;\dot{H}^{1+r})}    
  \end{align*}
  for all $\sigma \in [t_{j-1}, t_j]$. In the light of \eqref{eq2:reg} this
  proves \eqref{eq4:thetai} with $i = 1$. 

  By following the same steps, it holds for $\Theta_2$
  \begin{align*}
    &\big\| \Theta_2 (\sigma,t_{j-1})\big\|_{L_p(\Omega;\dot{H}^{-1})}\\
    &\quad \le C \sup_{x \in H} \big\| f'(x) \big\|_{\LB(H,\dot{H}^{-1+r})} 
    \Big\| \int_{t_{j-1}}^{\sigma} S(\sigma - \tau)
    f(X(\tau)) \diff{\tau} \Big\|_{L_p(\Omega;H)}.
  \end{align*}
  Now, by applying the fact that $$\| A^{\frac{1-r}{2}} S(\sigma - \tau)
  \|_{\LB(H)} \le C (\sigma - \tau)^{-\frac{1-r}{2}}, \quad \text{for all }
  t_{j-1} \le \tau < \sigma \le t_j,$$ we get for every $\sigma 
  \in [t_{j-1},t_j]$
  \begin{align*}
    \Big\| \int_{t_{j-1}}^{\sigma} S(\sigma - \tau)
    f(X(\tau)) \diff{\tau} \Big\|_{L_p(\Omega;H)}&\le C \int_{t_{j-1}}^{\sigma}
    (\sigma - \tau)^{-\frac{1-r}{2}} \| f(X(\tau))
    \|_{L_p(\Omega;\dot{H}^{-1+r})} \diff{\tau} \\
    &\le C \Big( 1 + \sup_{\sigma \in [0,T]} \big\| X(\sigma)
    \big\|_{L_p(\Omega;H)} \Big) k^{\frac{1+r}{2}}.
  \end{align*}
  As for $\Theta_1$ we therefore conclude 
  \begin{align*}
    \big\| \Theta_2 (\sigma,t_{j-1})\big\|_{L_p(\Omega;\dot{H}^{-1})} \le C
    k^{\frac{1+r}{2}} \quad \text{ for all } \sigma \in [t_{j-1},t_j].
  \end{align*}
  For the estimate of $\Theta_3$ we first apply the fact that
  \begin{align*}
    \E \Big[ \int_{0}^{1} f'(X(t_{j-1})) 
    \Big[  \int_{t_{j-1}}^{\sigma} S(\sigma - \tau)
    g(X(\tau))  \diff{W(\tau)} \Big] \diff{s} \Big| \F_{t_{j-1}} \Big] = 0.
  \end{align*}
  From this we get 
  \begin{align*}
    &\big\| \Theta_3 (\sigma,t_{j-1})\big\|_{L_p(\Omega;\dot{H}^{-1})}\\
    &\quad \le \int_{0}^{1} \Big\| \big( f'(\sigma,t_j,s) - f'(X(t_{j-1}))
    \big) \Big[ \int_{t_{j-1}}^{\sigma} S(\sigma - \tau)
    g(X(\tau))  \diff{W(\tau)} \Big]\Big\|_{L_p(\Omega;\dot{H}^{-1})}
    \diff{s}
  \end{align*}
  Further, for every $s \in [0,1]$ we derive by H\"older's inequality
  \begin{align*}
    &\Big\| \big( f'(\sigma,t_j,s) - f'(X(t_{j-1}))
    \big) \Big[ \int_{t_{j-1}}^{\sigma} S(\sigma - \tau)
    g(X(\tau))  \diff{W(\tau)} \Big]\Big\|_{L_p(\Omega;\dot{H}^{-1})}\\
    &\quad \le \Big( \E \Big[ \big\| f'(\sigma,t_j,s) - f'(X(t_{j-1}))
    \big\|_{\LB(H,\dot{H}^{-1})}^p 
    \Big\| \int_{t_{j-1}}^{\sigma} S(\sigma - \tau) g(X(\tau))  \diff{W(\tau)}
    \Big\|^p \Big] \Big)^{\frac{1}{p}}\\
    &\quad \le \big\| f'( X(t_{j-1}) + s (X(\sigma) - X(t_{j-1}))) -
    f'(X(t_{j-1})) 
    \big\|_{L_{2p}(\Omega;\LB(H,\dot{H}^{-1}))} \\
    &\qquad \times \Big\| \int_{t_{j-1}}^{\sigma}
    S(\sigma - \tau) g(X(\tau))  \diff{W(\tau)} \Big\|_{L_{2p}(\Omega;H)}.
  \end{align*}
  Now, we have by \eqref{eq4:flip} and \eqref{eq2:hoelder}
  \begin{align*}
    &\big\| f'( X(t_{j-1}) + s (X(\sigma) - X(t_{j-1}))) -
    f'(X(t_{j-1})) 
    \big\|_{L_{2p}(\Omega;\LB(H,\dot{H}^{-1}))}\\ 
    &\quad \le C_f \big\| X(\sigma) -
    X(t_{j-1}) \big\|_{L_{2p}(\Omega;H)} \le C (\sigma - t_{j-1})^{\frac{1}{2}}
  \end{align*}
  for all $s \in [0,1]$ and $\sigma \in [t_{j-1},t_j]$. In addition, by
  Proposition \ref{prop:stoch_int} it holds true that
  \begin{align*}
    &\Big\| \int_{t_{j-1}}^{\sigma}
    S(\sigma - \tau) g(X(\tau))  \diff{W(\tau)} \Big\|_{L_{2p}(\Omega;H)}\\
    &\quad \le C \Big( \int_{t_{j-1}}^{\sigma} \big\| S(\sigma - \tau)
    g(X(\tau)) \big\|^2_{L_p(\Omega;H)} \diff{\tau} \Big)^{\frac{1}{2}}
    \le C \Big( 1 + \sup_{\tau \in [0,T]} \big\| X(\tau)
    \big\|_{L_p(\Omega;H)} \Big) k^{\frac{1}{2}}
  \end{align*}
  for all $\sigma \in [t_{j-1},t_j]$. This completes the estimate of $\Theta_3$
  and, therefore, also the proof of the Lemma. 
\end{proof}

The last building block in the proof of consistency is the following lemma.

\begin{lemma}
  \label{lem4:nonlin2}
  Let Assumptions \ref{as4:initial} to \ref{as4:g} be satisfied for some $r \in
  [0,1)$. If the spatial discretization satisfies Assumption \ref{as:Vh} it
  holds 
  \begin{align*}
    &\max_{n \in \{1,\ldots,N_k\}} \Big\| \sum_{j = 1}^{n}
    S_{k,h}^{n-j +1 } \int_{t_{j-1}}^{t_j} \Big( g(X(\sigma)) - g(X(t_{j-1}))\\
    &\qquad \qquad \qquad - g'(X(t_{j-1}))\Big[ \int_{t_{j-1}}^{\sigma}
    g(X(t_{j-1})) \diff{W(\tau)}\Big] \Big)  
    \diff{W(\sigma)}  \Big\|_{L_p(\Omega;H)} \le C k^{\frac{1+r}{2}}   
  \end{align*}
  for all $h \in (0,1]$ and $k \in (0,T]$.
\end{lemma}

\begin{proof}
  The proof mainly applies the same techniques as used in the proof of Lemma
  \ref{lem4:nonlin2}. First let us fix an arbitrary $n \in \{1,\ldots,N_k\}$
  and recall the notation $\Gamma_X$ in \eqref{eq4:Gamma}. 
  Then we note that
  \begin{align*}
    &\E \Big[ S_{k,h}^{n-j+1} \int_{t_{j-1}}^{t_j}  g(X(\sigma)) -
    g(X(t_{j-1}))
    - g'(X(t_{j-1}))\big[ \Gamma_X(\sigma) \big]  
    \diff{W(\sigma)} \Big| \F_{t_{j-1}} \Big] = 0
  \end{align*}
  for every $j \in \{1,\ldots,n\}$. Hence, the sum of these terms is a discrete
  time martingale in $L_p(\Omega;H)$. Hence, we first apply Burkholder's
  inequality \cite[Th.~3.3]{burkholder1991} and then Proposition
  \ref{prop:stoch_int} and obtain
  \begin{align*}
    &\Big\| \sum_{j = 1}^{n}
    S_{k,h}^{n-j +1 } \int_{t_{j-1}}^{t_j} g(X(\sigma)) - g(X(t_{j-1}))
    - g'(X(t_{j-1}))\big[ \Gamma_X(\sigma) \big]   
    \diff{W(\sigma)}  \Big\|_{L_p(\Omega;H)}\\
    &\quad \le C \Big( \E \Big[ \Big( \sum_{j = 1}^n \Big\| S_{k,h}^{n-j+1}
    \int_{t_{j-1}}^{t_j} \Big( g(X(\sigma)) - g(X(t_{j-1}))
    \\ &\quad \qquad \qquad \qquad 
    - g'(X(t_{j-1}))\big[ \Gamma_X(\sigma) \big]
    \Big) \diff{W(\sigma)} \Big\|^{2} \Big)^{\frac{p}{2}} \Big]
    \Big)^{\frac{1}{p}} \\  
    &\quad \le C \Big( \sum_{j = 1}^n \Big\| 
    \int_{t_{j-1}}^{t_j} \Big( g(X(\sigma)) - g(X(t_{j-1}))
    \\ &\quad \qquad \qquad \qquad
    - g'(X(t_{j-1}))\big[ \Gamma_X(\sigma) \big]
    \Big) \diff{W(\sigma)} \Big\|^{2}_{L_p(\Omega;H)} \Big)^{\frac{1}{2}}\\
    &\quad \le C \Big( \sum_{j = 1}^n \int_{t_{j-1}}^{t_j} \big\| 
    g(X(\sigma)) - g(X(t_{j-1}))
    - g'(X(t_{j-1}))\big[ \Gamma_X(\sigma) \big]
    \big\|^{2}_{L_p(\Omega;\LB_2^0)} \diff{\sigma} \Big)^{\frac{1}{2}}. 
  \end{align*}
  Consequently, if we show that there exists a constant such that
  \begin{align}
    \label{eq4:normg}
    \big\| g(X(\sigma)) - g(X(t_{j-1})) - g'(X(t_{j-1}))\big[ \Gamma_X(\sigma)
    \big] \big\|^{2}_{L_p(\Omega;\LB_2^0)} \le C k^{1+r}
  \end{align}
  for all $\sigma \in [t_{j-1},t_j]$, the proof is complete. In order to prove
  \eqref{eq4:normg} we again apply the mean value theorem for Fr\'echet
  differentiable mappings and obtain 
  \begin{align*}
    g(X(\sigma)) - g(X(t_{j-1})) = \int_0^1 g'(\sigma,t_{j-1},s) \big[
    X(\sigma) - X(t_{j-1}) \big] \diff{s}, 
  \end{align*}
  where we denote
  \begin{align*}
    g'(\tau_1,\tau_2,s) := g'( X(\tau_2) + s (X(\tau_1) - X(\tau_2)))\quad
    \text{ for all } \tau_1, \tau_2 \in [0,T],\; s \in [0,1]. 
  \end{align*}
  After inserting
  \eqref{eq1:mild} we get the estimate
  \begin{align*}
    & \big\| g(X(\sigma)) - g(X(t_{j-1})) - g'(X(t_{j-1}))\big[
    \Gamma_X(\sigma) \big] \big\|_{L_p(\Omega;\LB_2^0)}\\
    &\quad \le \int_0^1 \big\| g'(\sigma,t_{j-1},s) \big[ \big( S(\sigma -
    t_{j-1}) - \Id_H \big) X(t_{j-1}) \big] \big\|_{L_p(\Omega;\LB_2^0)}
    \diff{s}\\
    &\qquad + \int_0^1 \Big\| g'(\sigma,t_{j-1},s) \Big[
    \int_{t_{j-1}}^{\sigma}S(\sigma - \tau) f(X(\tau))
    \diff{\tau}\Big] \Big\|_{L_p(\Omega;\LB_2^0)} \diff{s}\\
    &\qquad + \int_0^1 \Big\| g'(\sigma,t_{j-1},s) \Big[
    \int_{t_{j-1}}^{\sigma} S(\sigma - \tau) g(X(\tau))
    \diff{W(\tau)} - \Gamma_X(\sigma)\Big] \Big\|_{L_p(\Omega;\LB_2^0)}
    \diff{s}\\ 
    &\qquad + \int_0^1 \Big\| \big( g'(\sigma,t_{j-1},s) - g'(X(t_{j-1})) \big)
    \big[ \Gamma_X(\sigma) \big] \Big\|_{L_p(\Omega;\LB_2^0)} \diff{s}\\
    &\quad =: J_5 + \ldots + J_8.
  \end{align*}
  We consider the terms $J_i$, $i \in \{5,\ldots,8\}$, one by one. The desired
  estimated of $J_5$ is obtained in the same
  way as for the term $\Theta_1$ in the proof of Lemma \ref{lem4:nonlin1},
  namely
  \begin{align*}
    J_5  &\le \int_0^1 \big( \E \big[  \big\| g'(\sigma,t_{j-1},s)
    \big\|_{\LB(H,\LB_2^0)}^p  \big\| \big( S(\sigma -
    t_{j-1}) - \Id_H \big) X(t_{j-1}) \big\|^p \big]
    \big)^{\frac{1}{p}} \diff{s}\\
    &\le C \sup_{x \in H} \big\| g(x) \big\|_{\LB(H,\LB_2^0)}
    \big\| \big( S(\sigma - t_{j-1}) - \Id_H \big) X(t_{j-1})
    \big\|_{L_p(\Omega;H)}\\
    &\le C \sup_{x \in H} \big\| g(x) \big\|_{\LB(H,\LB_2^0)}
    \sup_{\tau \in [0,T]} \| X(\tau) \|_{L_p(\Omega;\dot{H}^{1+r})}
    k^{\frac{1+r}{2}}.
  \end{align*}
  Likewise, the estimate of $J_6$ is done by the exact same steps as for the
  term $\Theta_2$ in the proof of Lemma \ref{lem4:nonlin1}. Thus, it holds
  \begin{align*}
    J_6 \le C \sup_{x \in H} \big\| g(x) \big\|_{\LB(H,\LB_2^0)} \Big( 1 +
    \sup_{\sigma \in [0,T]} \big\| X(\sigma) \big\|_{L_p(\Omega;H)} \Big)
    k^{\frac{1+r}{2}}.
  \end{align*}
  As above, the term $J_7$ is first estimated by
  \begin{align*}
    J_7&\le C \sup_{x \in H} \big\| g(x) \big\|_{\LB(H,\LB_2^0)} 
    \Big\| \int_{t_{j-1}}^{\sigma} S(\sigma - \tau) g(X(\tau))
    \diff{W(\tau)} - \Gamma_X(\sigma)\Big\|_{L_p(\Omega;H)}
  \end{align*}
  Then, after inserting the definition \eqref{eq4:Gamma} of $\Gamma_X$ and an
  application of Proposition \ref{prop:stoch_int} we arrive at
  \begin{align*}
    &\Big\| \int_{t_{j-1}}^{\sigma} S(\sigma - \tau) g(X(\tau))
    \diff{W(\tau)} - \Gamma_X(\sigma)\Big\|_{L_p(\Omega;H)}\\ 
    &\quad \le C \Big( \int_{t_{j-1}}^{\sigma} 
    \big\|  S(\sigma - \tau) g(X(\tau)) - g(X(t_{j-1}))
    \big\|_{L_p(\Omega;\LB_2^0)}^2 \diff{\tau} \Big)^{\frac{1}{2}}\\
    &\quad \le C \Big( \int_{t_{j-1}}^{\sigma} 
    \big\|  \big( S(\sigma - \tau) - \Id_H \big) g(X(\tau)) 
    \big\|_{L_p(\Omega;\LB_2^0)}^2 \diff{\tau} \Big)^{\frac{1}{2}}\\
    &\qquad + C \Big( \int_{t_{j-1}}^{\sigma} \big\|    
    g(X(\tau)) - g(X(t_{j-1})) \big\|_{L_p(\Omega;\LB_2^0)}^2 \diff{\tau}
    \Big)^{\frac{1}{2}}.
  \end{align*}
  For the first summand recall by \eqref{eq4:glin} that $g$ yields some
  additional spatial regularity. Together with \cite[Ch.~2.6,
  Th.~6.13]{pazy1983} we can use this to obtain
  \begin{align*}
    & \Big( \int_{t_{j-1}}^{\sigma} 
    \big\|  \big( S(\sigma - \tau) - \Id_H \big) g(X(\tau)) 
    \big\|_{L_p(\Omega;\LB_2^0)}^2 \diff{\tau} \Big)^{\frac{1}{2}}\\
    &\quad \le C \Big( \int_{t_{j-1}}^{\sigma} (\sigma - \tau)^{r}  
    \big\| g( X(\tau) ) \big\|_{L_p(\Omega;\LB_{2,r}^0)}^2 \diff{\tau}
    \Big)^{\frac{1}{2}} \\
    &\quad \le C \Big( 1 + \sup_{\tau \in [0,T]} \big\| X(\tau)
    \big\|_{L_p(\Omega;\dot{H}^{r})} \Big) k^{\frac{1+r}{2}}
  \end{align*}
  for all $\sigma \in [t_{j-1}, t_j]$. A similar estimate follows for the
  second summand by \eqref{eq4:glip} and \eqref{eq2:hoelder}, that is
  \begin{align*}
    \Big( \int_{t_{j-1}}^{\sigma} \big\|    
    g(X(\tau)) - g(X(t_{j-1})) \big\|_{L_p(\Omega;\LB_2^0)}^2 \diff{\tau}
    \Big)^{\frac{1}{2}} \le C \Big( \int_{t_{j-1}}^{\sigma} ( \tau - t_{j-1})
    \diff{\tau} \Big)^{\frac{1}{2}} \le C k^{\frac{1+r}{2}}. 
  \end{align*}
  This shows the desired estimate for $J_7$ and it remains to consider $J_8$.
  The estimate of $J_8$ is very similar to the estimate of $\Theta_3$ in the
  proof of Lemma \ref{lem4:nonlin1}. After the application of H\"older's
  inequality we arrive at
  \begin{align*}
    J_8 &\le \int_{0}^{1} \big\| g'\big( X(t_{j-1}) +  s \big( X(\sigma) -
    X(t_{j-1}) \big) \big) - g'(X(t_{j-1}))
    \big\|_{L_{2p}(\Omega;\LB(H,\LB_2^0))}\diff{s}\\
    &\qquad \times \Big\| \int_{t_{j-1}}^{\sigma} g(X(t_{j-1})) \diff{W(\tau)}
    \Big\|_{L_{2p}(\Omega;H)}
  \end{align*}
  Next, by \eqref{eq4:glip} and \eqref{eq2:hoelder} it follows
  \begin{align*}
    \big\| g'\big( X(t_{j-1}) +  s \big( X(\sigma) -
    X(t_{j-1}) \big) \big) - g'(X(t_{j-1}))
    \big\|_{L_{2p}(\Omega;\LB(H,\LB_2^0))} \le C (\sigma -
    t_{j-1})^{\frac{1}{2}}, 
  \end{align*}
  while Proposition \ref{prop:stoch_int} and Assumption \ref{as4:g} yield
  \begin{align*}
    \Big\| \int_{t_{j-1}}^{\sigma} g(X(t_{j-1})) \diff{W(\tau)}
    \Big\|_{L_{2p}(\Omega;H)} &\le C \Big( \int_{t_{j-1}}^{\sigma} \big\|
    g(X(t_{j-1})) \big\|_{L_{2p}(\Omega;\LB_2^0)}^2 \diff{\tau}
    \Big)^{\frac{1}{2}}\\
    &\le C \Big( 1 + \sup_{\tau \in [0,T]} \| X(\tau)
    \|_{L_{2p}(\Omega;H)}\Big) k^{\frac{1}{2}},
  \end{align*}
  for all $\sigma \in [t_{j-1},t_j]$. Therefore, there exists a constant such
  that 
  \begin{align*}
    J_8 \le C k
  \end{align*}
  and the assertion of the lemma has been proved.
\end{proof}


\section{Noise approximation}
\label{sec:Noise}

Starting point of the spectral noise approximation is the covariance operator
$Q \in \LB(U)$, which is symmetric and nonnegative. Since we do not assume that
$Q$ has finite trace, we need to approximate the stochastic integral with
respect to a \emph{cylindrical} $Q$-Wiener process $W \colon [0,T] \times
\Omega \to U$ (see \cite[Ch.~2.5]{roeckner2007}).

Throughout this section we work under the assumption that there exists
an orthonormal basis $(\varphi_j)_{j \in \N}$ of the separable Hilbert space
$U$ such that 
\begin{align*}
  Q \varphi_j = \mu_j \varphi_j, \quad \text{ for all } j \in \N,
\end{align*}
where $\mu_j \ge 0$, $j \in \N$, denote the eigenvalues of $Q$. First let us
note that this assumption is not fulfilled for all symmetric and
nonnegative operators $Q \in L(U)$. However, it always holds true for white
noise, that is $Q = \Id_U$, or if $Q$ is of finite trace by the
spectral theorem for compact, symmetric operators. Further, the family
$(\sqrt{\mu_j} \varphi_j)_{j \in \N}$ is an orthonormal basis of the
Cameron-Martin space $U_0 = Q^{\frac{1}{2}}(U)$, which 
is endowed with the inner product $(u,v)_{U_0} := ( Q^{-\frac{1}{2}} u,
Q^{-\frac{1}{2}} v )_{U}$ for all $u,v \in U_0$ (see
\cite[Ch.~2.3]{roeckner2007}). 

As demonstrated in \cite[Rem.~2.5.1, Prop.~2.5.2]{roeckner2007}, in order to
define the stochastic integral with respect to a cylindrical Wiener process,
one introduces a further Hilbert space $U_1$ and an Hilbert-Schmidt embedding
$\mathcal{I} \colon U_0 \to U_1$, such that $W$ becomes a standard Wiener
process on the larger space $U_1$ with covariance operator $Q_1 :=
\mathcal{I}\mathcal{I}^{\ast}$ and Karhunen-Lo\`eve expansion 
\begin{align}
  \label{eq5:noiseKL}
  W(t) = \sum_{j = 1}^\infty  \beta_j(t) \mathcal{I}(\sqrt{\mu_j} \varphi_j),
  \quad t \in [0,T],
\end{align}
where $\beta_j \colon [0,T] \times \Omega \to \R$, $j \in \N$, is a family of
independent, standard real-valued Brownian motions. Since $\mathcal{I} \colon 
U_0 \to Q_1^{\frac{1}{2}}(U_0)$ is an isometry, the definition of the
stochastic integral with respect to a cylindrical Wiener process is in fact
independent of the choice of the space $U_1$, see
\cite[Rem.~2.5.3]{roeckner2007}. Finally note that one can choose the
Hilbert-Schmidt operator $\mathcal{I}$ in such a way that $(
\mathcal{I}(\sqrt{\mu_j} \varphi_j))_{j\in \N}$ becomes an orthonormal basis of
$Q_1^{\frac{1}{2}}(U_1)$.   

In order to approximate the Wiener process we follow  
the footprints of \cite{barth2011}. Let us denote by $Q_J \in \LB(U)$, $J \in
\N$, the operator given by  
\begin{align*}
  Q_J \varphi_j :=
  \begin{cases}
    \mu_j \varphi_j,& \text{ if } j \in \{1,\ldots,J\},\\
    0,& \text{ else}.
  \end{cases}
\end{align*}
As in \cite{barth2011} we further use the abbreviation $Q_{cJ} := Q - Q_J$.
Now, since $Q_J$ is of finite rank, the $Q_J$-Wiener process $W^J \colon [0,T]
\times \Omega \to U$ defined by
\begin{align}
  \label{eq5:noiseexp}
  W^J(t) = \sum_{j = 1}^J \sqrt{\mu_j} \beta_j(t) \varphi_j, \quad t \in [0,T],
\end{align}
can be simulated on a computer, provided that the orthonormal basis
$(\varphi_j)_{j \in \N}$ of $U$ is explicitly known. Here $\beta_j \colon [0,T]
\times \Omega \to \R$, $j \in \N$, are the same as in \eqref{eq5:noiseKL}
Further, from \cite{barth2011} we recall the notation $W^{cJ}(t) := W(t) -
W^J(t)$ for all $t \in [0,T]$. 

Then, the Milstein-Galerkin finite element scheme with truncated noise is given
by the recursion 
\begin{align}
  \label{eq5:truncMilstein}
  \begin{split}
    X_{k,h,J}(t_0) &= P_h X_0,\\
    X_{k,h,J}(t_n) &= X_{k,h,J}(t_{n-1}) - k \big[ A_h X_{k,h,J}(t_n) +
    P_h f(X_{k,h,J}(t_{n-1})) \big]\\
    &\quad + P_h g(X_{k,h,J}(t_{n-1})) \Delta_k W^{J^2}(t_{n})
    \\  
    &\quad + \int_{t_{n-1}}^{t_n} P_h g'(X_{k,h,J}(t_{n-1}))\Big[ 
    \int_{t_{n-1}}^{\sigma_1} g(X_{k,h,J}(t_{n-1})) \diff{W^J(\sigma_2)} \Big]
    \diff{W^J(\sigma_1)}
  \end{split}
\end{align}
for $n \in \{1,\ldots,N_k\}$ and all $h \in (0,1]$, $k \in (0,T]$ and $J \in
\N$. We stress that the
Euler-Maruyama term incorporates $J^2$ summands of the Wiener noise expansion
\eqref{eq5:noiseexp} while the additional iterated integral term of the
Milstein scheme only uses $J$ summands. As discussed in \cite{barth2011} this
leads to an optimal balance between computational cost and order of
convergence. 

First we embed the scheme \eqref{eq5:truncMilstein} into the abstract framework
of Section \ref{sec:numscheme}. Compared to the scheme \eqref{eq4:Milstein} the
only difference appears in the definition of the increment function, which is
now given by
\begin{align}
  \label{eq5:PhihJ}
  \begin{split}
    \Phi_{h,J}(x,t,k) &= - k S_{k,h} f(x) + S_{k,h} g(x) \big(W^{J^2}(t+k) -
    W^{J^2}(t) \big)\\
    &\qquad + S_{k,h} \int_{t}^{t + k} g'(x)\Big[ \int_{t}^{\sigma_1} g(x)
    \diff{W^J(\sigma_2)} \Big] \diff{W^J(\sigma_1)}, 
  \end{split}
\end{align}
for $(t,k) \in \mathbb{T}$. Our first result is concerned with the stability of
the Milstein scheme with truncated noise. 

\begin{theorem}
  \label{th:truncMilstab}
  Under Assumptions \ref{as4:initial} to \ref{as4:g} the Milstein-Galerkin
  finite element scheme \eqref{eq5:truncMilstein} is bistable for every $h \in
  (0,1]$ and $J \in \N$. The stability constant $C_{\mathrm{Stab}}$ can be
  chosen to be independent of $h \in (0,1]$ and $J \in \N$. 
\end{theorem}

\begin{proof}
  We only need to verify that Assumption \ref{as:stab} is also satisfied by
  $\Phi_{h,J}$. This is done by the exact same steps as in the proof of Theorem
  \ref{th:Milstab}. An important tool in that proof is Proposition
  \ref{prop:stoch_int}. Here we have to apply it to stochastic integrals with
  respect to $W^J$. For this, let $\Psi \colon [0,T] \times \Omega \to \LB_2^0$
  be a predictable stochastic process satisfying the condition of Proposition
  \ref{prop:stoch_int}, then it holds
  \begin{align*}
    \Big\| \int_{\tau_1}^{\tau_2} \Psi(\sigma) \diff{W^J(\sigma)}
    \Big\|_{L_p(\Omega;H)} \le C(p) \Big( \E \Big[ \Big( \int_{\tau_1}^{\tau_2}
    \big\| \Psi(\sigma) Q_J^{\frac{1}{2}} \big\|_{\LB_2(U,H)}^2 \diff{\sigma}
    \Big)^{\frac{p}{2}} \Big] \Big)^{\frac{1}{p}},
  \end{align*}
  for all $0 \le \tau_1 < \tau_2 \le T$ and all $J \in \N$. Since we have
  \begin{align*}
    \big\| \Psi(\sigma) Q_J^{\frac{1}{2}} \big\|_{\LB_2(U,H)}^2 = \sum_{j = 1}^J
    \mu_j \big\| \Psi(\sigma) \varphi_j \big\|^2 \le \sum_{j = 1}^\infty
    \mu_j \big\| \Psi(\sigma) \varphi_j \big\|^2 = \big\| \Psi(\sigma)
    \big\|_{\LB_2^0}^2,
  \end{align*}
  the $L_p$-norm of the stochastic integral with respect to $W^J$ is therefore
  bounded by 
  \begin{align*}
    \Big\| \int_{\tau_1}^{\tau_2} \Psi(\sigma) \diff{W^J(\sigma)}
    \Big\|_{L_p(\Omega;H)} \le C(p) \Big( \E \Big[ \Big( \int_{\tau_1}^{\tau_2}
    \big\| \Psi(\sigma) \big\|_{\LB_2^0}^2 \diff{\sigma}
    \Big)^{\frac{p}{2}} \Big] \Big)^{\frac{1}{p}}.  
  \end{align*}
  In particular, the constant is independent of $J \in \N$. Keeping this in
  mind, all steps in the proof of Theorem \ref{th:Milstab} remain also valid
  for \eqref{eq5:truncMilstein}.  
\end{proof}

Having this established it remains to investigate if 
the scheme \eqref{eq5:truncMilstein} is consistent. For this we introduce the
following additional condition, which allows us to control the order of
consistency with respect to the parameter $J \in \N$. 

\begin{assumption}
  \label{as5:noise}
  There exist constants $C$ and $\alpha > 0$ such that
  \begin{align}
    \label{eq5:alpha1}
    \Big( \sum_{j = 1}^\infty j^{\alpha} \mu_j \big\| g(x) \varphi_j\big\|^2
    \Big)^{\frac{1}{2}} &\le C \big( 1 + \| x \| \big),\\
    \label{eq5:alpha2}
    \Big( \sum_{j = 1}^\infty j^{\alpha} \mu_j \big\| g'(x)[y]
    \varphi_j\big\|^2 \Big)^{\frac{1}{2}} &\le C \| y \| 
  \end{align}
  for all $x,y \in H$, where $(\varphi_j, \mu_j)_{j \in \N}$ are the eigenpairs
  of the covariance operator $Q$.
\end{assumption}

\begin{remarks}
    1) Note that \eqref{eq5:alpha1} and \eqref{eq5:alpha2} coincide if $g
    \colon H \to \LB_2^0$ is linear. 

    2) Assumption \ref{as4:g} ensures that \eqref{eq5:alpha1} and
      \eqref{eq5:alpha2} are fulfilled with $\alpha = 0$. Further, provided
      that $Q$ is of finite trace and $g \colon H \to \LB_2^0$ satisfies
      \begin{align*}
	\| g (x) \|_{\LB(U,H)} &\le C \big( 1 + \| x \| \big), \qquad \text{
	and } \| g' (x)[y] \|_{\LB(U,H)} \le C \|y\|  
      \end{align*}
      for all $x,y\in H$, then \eqref{eq5:alpha1} and \eqref{eq5:alpha2}
      simplify to 
      \begin{align}
	\label{eq5:alpha3}
	\Big( \sum_{j = 1}^\infty j^{\alpha} \mu_j \Big)^{\frac{1}{2}} <
	\infty.
      \end{align}
      Hence, the order of convergence of the truncated
      noise only depends on the rate of decay of the eigenvalues of the
      covariance operator $Q$.
\end{remarks}

\begin{theorem}
  \label{th5:trunccons}
  Let Assumptions \ref{as:Vh} and \ref{as:Vh2} be fulfilled by 
  the spatial discretization.  
  If Assumptions \ref{as4:initial} to \ref{as4:g} and Assumption
  \ref{as5:noise} are satisfied with $p \in [2,\infty)$, 
  $r \in [0,1)$ and $\alpha > 0$, then there exists a constant $C$ such that 
  the following estimate holds true for the local truncation error of the
  Milstein scheme \eqref{eq5:truncMilstein}, namely
  \begin{align*}
    \big\| \mathcal{R}_k \big[ X|_{\mathcal{T}_k} \big] \big\|_{-1,p} \le C
    \big( h^{1+r} + k^{\frac{1+r}{2}} + J^{-\alpha} \big)
  \end{align*}
  for all $h \in [0,1)$, $k \in (0,T]$ and $J \in \N$. In particular, if $h$,
  $J$ and $k$ are coupled by $h := c_1 k^{\frac{1}{2}}$ and $J := \lceil c_2
  k^{-\frac{1+r}{2\alpha}}\rceil$ for some positive constants $c_1,c_2 \in \R$,
  then the Milstein scheme is consistent of order 
  $\frac{1+r}{2}$.
\end{theorem}

\begin{proof}
  The proof relies on slightly generalized techniques from \cite[Lem.~4.1 and
  Lem.~4.2]{barth2011}. First let us note, that the results of Lemmas
  \ref{lem4:initial} to \ref{lem4:nonlin1} remain valid for
  \eqref{eq5:truncMilstein}. Therefore, from the decomposition of the local
  truncation error in Lemma \ref{lem:cons} and \eqref{eq4:cons5} it follows
  that we only need to show that the following analogue of Lemma
  \ref{lem4:nonlin2} is valid: There exists a constant $C$ such that for all $n
  \in \{1,\ldots,N_k\}$ 
  \begin{align*}
    &\Big\| \sum_{j = 1}^{n}
    S_{k,h}^{n-j +1 } \Big( \int_{t_{j-1}}^{t_j} g(X(\sigma)) \diff{W(\sigma)}
    - \int_{t_{j-1}}^{t_j} g(X(t_{j-1})) \diff{W^{J^2}(\sigma)}\\
    &\qquad \qquad - \int_{t_{j-1}}^{t_j} g'(X(t_{j-1}))\Big[
    \int_{t_{j-1}}^{\sigma} g(X(t_{j-1})) \diff{W^J(\tau)}\Big]   
    \diff{W^J(\sigma)} \Big)  \Big\|_{L_p(\Omega;H)}
    \\ &\quad 
    \le C \big( k^{\frac{1+r}{2}}
    + J^{-\alpha} \big)
  \end{align*}
  for all $h \in (0,1]$, $k \in (0,T]$ and $J \in \N$.  We begin by fixing an
  arbitrary $n \in \{1,\ldots,N_k\}$ and insert several suitable terms with
  untruncated noise such that Lemma \ref{lem4:nonlin2} is applicable. Hence, we
  obtain 
  \begin{align}
    \label{eq5:terms}
    \begin{split}
    &\Big\| \sum_{j = 1}^{n}
    S_{k,h}^{n-j +1 } \Big( \int_{t_{j-1}}^{t_j} g(X(\sigma)) \diff{W(\sigma)}
    - \int_{t_{j-1}}^{t_j} g(X(t_{j-1})) \diff{W^{J^2}(\sigma)}\\
    &\qquad \qquad - \int_{t_{j-1}}^{t_j} g'(X(t_{j-1}))\Big[
    \int_{t_{j-1}}^{\sigma} g(X(t_{j-1})) \diff{W^J(\tau)}\Big]   
    \diff{W^J(\sigma)} \Big) \Big\|_{L_p(\Omega;H)}\\
    &\quad \le C k^{\frac{1+r}{2}} + \Big\| \sum_{j = 1}^{n} S_{k,h}^{n-j +1 }
    \int_{t_{j-1}}^{t_j} g(X(t_{j-1})) \diff{W^{cJ^2}(\sigma)}
    \Big\|_{L_p(\Omega;H)} \\
    &\qquad + \Big\| \sum_{j = 1}^{n} S_{k,h}^{n-j +1 } \Big(
    \int_{t_{j-1}}^{t_j} g'(X(t_{j-1}))\Big[
    \int_{t_{j-1}}^{\sigma} g(X(t_{j-1})) \diff{W(\tau)}\Big]   
    \diff{W(\sigma)}\\ 
    &\qquad \qquad - \int_{t_{j-1}}^{t_j} g'(X(t_{j-1}))\Big[
    \int_{t_{j-1}}^{\sigma} g(X(t_{j-1})) \diff{W^J(\tau)}\Big]   
    \diff{W^J(\sigma)} \Big) \Big\|_{L_p(\Omega;H)}. 
    \end{split}
  \end{align}
  First let us note that by Assumption \ref{as5:noise} it holds
  \begin{align}
    \label{eq5:gHS}
    \begin{split}
    \big\| g(X(t_{j-1}))
    \big\|_{L_p(\Omega;\LB_2(Q_{cJ^2}^{\frac{1}{2}}(U),H))} &= \Big( \E \Big[
    \Big( \sum_{j = J^2+1}^\infty \mu_j \big\|  g(X(t_{j-1})) \varphi_j
    \big\|^2 \Big)^{\frac{p}{2}} \Big] \Big)^{\frac{1}{p}}\\
    &\le \Big( \E \Big[
    \Big( \sum_{j = J^2+1}^\infty \frac{j^\alpha}{J^{2\alpha}} \mu_j \big\|
    g(X(t_{j-1})) \varphi_j \big\|^2 \Big)^{\frac{p}{2}} \Big]
    \Big)^{\frac{1}{p}}\\ 
    &\le C J^{-\alpha} \Big( 1 + \sup_{t\in [0,T]} \big\| X(t)
    \big\|_{L_p(\Omega;H)}\Big).
    \end{split}
  \end{align}
  This together with Burkholder's inequality
  \cite[Th.~3.3]{burkholder1991}, \eqref{eq2:stabSkh} and Proposition
  \ref{prop:stoch_int} applied to $W^{cJ^2}$  yields for the second term
  \begin{align*}
    &\Big\| \sum_{j = 1}^{n} S_{k,h}^{n-j +1 }
    \int_{t_{j-1}}^{t_j} g(X(t_{j-1})) \diff{W^{cJ^2}(\sigma)}
    \Big\|_{L_p(\Omega;H)} \\
    &\quad \le C \Big( \E \Big[ \Big( \sum_{j = 1}^{n} \Big\| S_{k,h}^{n-j +1 }
    \int_{t_{j-1}}^{t_j} g(X(t_{j-1})) \diff{W^{cJ^2}(\sigma)} \Big\|^2
    \Big)^{\frac{p}{2}} \Big] \Big)^{\frac{1}{p}} \\
    &\quad \le C  \Big( \sum_{j = 1}^{n} \Big\| S_{k,h}^{n-j +1 }
    \int_{t_{j-1}}^{t_j} g(X(t_{j-1})) \diff{W^{cJ^2}(\sigma)}
    \Big\|^2_{L_p(\Omega;H)} \Big)^{\frac{1}{2}}\\
    &\quad \le C  \Big( \sum_{j = 1}^{n} \int_{t_{j-1}}^{t_j} \big\|
    g(X(t_{j-1})) \big\|^2_{L_p(\Omega;\LB_2(Q_{cJ^2}^{\frac{1}{2}}(U),H))}
    \diff{\sigma} \Big)^{\frac{1}{2}} \le C \sqrt{T} J^{-\alpha}.
  \end{align*}
  By the same arguments we get for the third summand in \eqref{eq5:terms} that
  \begin{align*}
    &\Big\| \sum_{j = 1}^{n} S_{k,h}^{n-j +1 } \Big(
    \int_{t_{j-1}}^{t_j} g'(X(t_{j-1}))\Big[
    \int_{t_{j-1}}^{\sigma} g(X(t_{j-1})) \diff{W(\tau)}\Big]   
    \diff{W(\sigma)}\\ 
    &\qquad \qquad - \int_{t_{j-1}}^{t_j} g'(X(t_{j-1}))\Big[
    \int_{t_{j-1}}^{\sigma} g(X(t_{j-1})) \diff{W^J(\tau)}\Big]   
    \diff{W^J(\sigma)} \Big) \Big\|_{L_p(\Omega;H)}\\
    &\quad \le C \Big( \sum_{j = 1}^{n} \Big\| \int_{t_{j-1}}^{t_j}
    g'(X(t_{j-1}))\Big[ \int_{t_{j-1}}^{\sigma} g(X(t_{j-1}))
    \diff{W^{cJ}(\tau)}\Big]
    \diff{W(\sigma)}\Big\|_{L_p(\Omega;H)}^2\Big)^{\frac{1}{2}}\\ 
    &\qquad + C \Big( \sum_{j = 1}^{n} \Big\|  \int_{t_{j-1}}^{t_j}
    g'(X(t_{j-1}))\Big[ 
    \int_{t_{j-1}}^{\sigma} g(X(t_{j-1})) \diff{W(\tau)}\Big]   
    \diff{W^{cJ}(\sigma)} \Big\|_{L_p(\Omega;H)}^2
    \Big)^{\frac{1}{2}}.
  \end{align*}
  Then, by two applications of Proposition \ref{prop:stoch_int} and Assumptions
  \ref{as4:g} and \eqref{eq5:gHS} it follows
  \begin{align*}
    &\Big( \sum_{j = 1}^{n} \Big\| \int_{t_{j-1}}^{t_j}
    g'(X(t_{j-1}))\Big[ \int_{t_{j-1}}^{\sigma} g(X(t_{j-1}))
    \diff{W^{cJ}(\tau)}\Big]
    \diff{W(\sigma)}\Big\|_{L_p(\Omega;H)}^2\Big)^{\frac{1}{2}}\\
    &\quad \le C \Big( \sum_{j = 1}^{n} \int_{t_{j-1}}^{t_j} \Big\| 
    g'(X(t_{j-1})) \Big[ \int_{t_{j-1}}^{\sigma} g(X(t_{j-1}))
    \diff{W^{cJ}(\tau)}\Big] \Big\|_{L_p(\Omega;\LB_2^0)}^2
    \diff{\sigma} \Big)^{\frac{1}{2}}\\
    &\quad\le C \sup_{x \in H} \big\| g'(x) \big\|_{\LB(H,\LB_2^0)} 
    \Big( \sum_{j = 1}^{n} \int_{t_{j-1}}^{t_j} \int_{t_{j-1}}^{\sigma} 
    \big\| g(X(t_{j-1})) \big\|_{L_p(\Omega;\LB_2(Q_{cJ}^{\frac{1}{2}},H))}^2
    \diff{\tau} \diff{\sigma} \Big)^{\frac{1}{2}}\\
    &\quad \le C C_g \sqrt{T}  k^{\frac{1}{2}}
    J^{-\frac{\alpha}{2}} \le C \big( k + J^{-\alpha} \big), 
  \end{align*}
  where we applied the inequality $ab \le \frac{1}{2} (a^2+ b^2)$ in the last
  step. By using \eqref{eq5:alpha2} we obtain the following estimate in the
  same way as in \eqref{eq5:gHS}, 
  \begin{align*}
    &\Big( \sum_{j = 1}^{n} \Big\|  \int_{t_{j-1}}^{t_j}
    g'(X(t_{j-1}))\Big[ 
    \int_{t_{j-1}}^{\sigma} g(X(t_{j-1})) \diff{W(\tau)}\Big]   
    \diff{W^{cJ}(\sigma)} \Big\|_{L_p(\Omega;H)}^2
    \Big)^{\frac{1}{2}}\\
    &\quad \le C \Big( \sum_{j = 1}^{n} \int_{t_{j-1}}^{t_j} \Big\|
    g'(X(t_{j-1}))\Big[ \int_{t_{j-1}}^{\sigma} g(X(t_{j-1}))
    \diff{W(\tau)}\Big] \Big\|_{L_p(\Omega;\LB_2(Q_{cJ}^{\frac{1}{2}}(U),H))}^2
    \diff{\sigma} \Big)^{\frac{1}{2}}  \\
    &\quad \le C \Big( \sum_{j = 1}^{n} \int_{t_{j-1}}^{t_j} J^{-\alpha}
    \Big\| \int_{t_{j-1}}^{\sigma} g(X(t_{j-1})) \diff{W(\tau)}
    \Big\|_{L_p(\Omega;H)}^2 \diff{\sigma} \Big)^{\frac{1}{2}}\\ 
    &\quad \le C J^{-\frac{\alpha}{2}} \Big( \sum_{j = 1}^{n}
    \int_{t_{j-1}}^{t_j} \big(\sigma - t_{j-1}\big) \big\| g(X(t_{j-1})) 
    \big\|_{L_p(\Omega;\LB_2^0)}^2   \diff{\sigma} \Big)^{\frac{1}{2}}\\
    &\quad \le C  \sqrt{T} C_g \Big( 1 + \sup_{t \in [0,T]} \big\| X(t)
    \big\|_{L_p(\Omega;H)} \Big) k^{\frac{1}{2}} J^{-\frac{\alpha}{2}}
    \le C \big( k + J^{-\alpha}\big).   
  \end{align*}
  This completes the proof.
\end{proof}

Combining Theorems \ref{th:truncMilstab} and \ref{th5:trunccons} with Theorem
\ref{th:lvlconv} immediately yields the following convergence result:

\begin{theorem}
  \label{th6:conv}
  Let Assumptions \ref{as:Vh} and \ref{as:Vh2} be fulfilled by 
  the spatial discretization.  
  If Assumptions \ref{as4:initial} to \ref{as4:g} and Assumption
  \ref{as5:noise} are satisfied with $p \in [2,\infty)$, 
  $r \in [0,1)$ and $\alpha > 0$, then there exists a constant $C$ such that
  \begin{align*}
    \max_{0 \le n \le N_k} \big\| X_{k,h,J}(t_n) - X(t_n)
    \big\|_{L_p(\Omega;H)}  \le C \big( h^{1+r} + k^{\frac{1+r}{2}} +
    J^{-\alpha} \big) 
  \end{align*}
  for all $h \in (0,1]$, $k \in (0,T]$ and $J \in \N$, where $X_{k,h,J}$
  denotes the grid function generated by the Milstein scheme with truncated
  noise \eqref{eq5:truncMilstein} and $X$ is the mild solution to
  \eqref{eq1:SPDE}.
\end{theorem}

\section*{Acknowledgement}
The author wants to express his gratitude to all his colleagues, who gave many
helpful advices and valuable insights during several encouraging and
inspiring discussions. In particular, my thanks goes to Adam Andersson, Andrea
Barth, Wolf-J\"urgen Beyn, Sonja Cox, Arnulf Jentzen, Annika Lang and Stig
Larsson. 

I also gratefully acknowledge the hospitality and the travel support by
Chalmers University of Technology, G\"oteborg, Sweden, as well as by
the Collaborative Research Centre 701 ‘Spectral
Structures and Topological Methods in Mathematics’ at Bielefeld
University, Germany.

\def\cprime{$'$} \def\polhk#1{\setbox0=\hbox{#1}{\ooalign{\hidewidth
  \lower1.5ex\hbox{`}\hidewidth\crcr\unhbox0}}}

\end{document}